\documentclass[14pt, dvipsnames]{extarticle}
\usepackage[left=2cm, top=2cm, right=2cm, bottom=20mm]{geometry} 

\usepackage[T1]{fontenc}
\usepackage[cmintegrals,cmbraces]{newtxmath}
\usepackage[lining]{ebgaramond}
\usepackage{ebgaramond-maths}
\usepackage{euscript}
\usepackage{xcolor}
\usepackage{graphicx}
\usepackage[figurename=Fig.]{caption}

\usepackage[tracking = true, letterspace = 50]{microtype}
\SetTracking{encoding=T1, shape=sc}{50}
\SetTracking{encoding=T1}{50}

\usepackage{amsthm} 
\usepackage{hyperref} 
\usepackage[hyphenbreaks]{breakurl}

\theoremstyle{definition}
\newtheorem{theorem}{Theorem}
\newtheorem{coroll}{Corollary}[theorem]
\newtheorem{lemma}{Lemma}
\newtheorem{prop}{Proposition}

\newtheorem*{remark*}{Remark}
\newtheorem*{theoremprime}{Theorem}

\theoremstyle{definition}
\newtheorem{defi}{Definition}

\newtheorem{example}{\bf Example}

\theoremstyle{remark}

\usepackage{enumerate}

\usepackage{etoolbox}
\patchcmd{\thmhead}{(#3)}{#3}{}{}

\let\phi\varphi

\DeclareMathOperator{\kk}{\Bbbk}

\DeclareMathOperator{\res}{\mathrm{res}}

\DeclareMathOperator{\Q}{\mathbb{Q}}

\DeclareMathOperator{\G}{\mathbb{G}}
\DeclareMathOperator{\id}{\mathrm{id}}
\DeclareMathOperator{\Sym}{\mathrm{Sym}}
\DeclareMathOperator{\nGrp}{\mathit{n}-\mathcal{G}\mathrm{rp}}
\DeclareMathOperator{\inv}{\mathrm{inv}}
\DeclareMathOperator{\Aut}{\mathrm{Aut}}

\DeclareMathOperator{\Z}{\mathbb{Z}}
\DeclareMathOperator{\Cc}{\mathbb{C}}

\DeclareMathOperator{\disc}{\mathrm{disc}}
\DeclareMathOperator{\R}{\mathbb{R}}

\DeclareMathOperator{\rep}{\mathrm{rep}}

\DeclareMathOperator{\vecc}{\mathrm{vec}}
\DeclareMathOperator{\tr}{\mathrm{tr}}
\DeclareMathOperator{\charr}{\mathrm{char}}
\DeclareMathOperator{\Toep}{\mathrm{Toep}}
\DeclareMathOperator{\sgn}{\mathrm{sgn}}
\DeclareMathOperator{\Circ}{\mathrm{Circ}}
\DeclareMathOperator{\GL}{\mathrm{GL}}
\DeclareMathOperator{\h}{\mathrm{h}}
\DeclareMathOperator{\diag}{\mathrm{diag}}

\DeclareMathOperator{\Mat}{\mathrm{Mat}}
\DeclareMathOperator{\E}{\mathcal{E}}

\makeatletter
  \DeclareSymbolFont{ntxletters}{OML}{ntxmi}{m}{it}
  \SetSymbolFont{ntxletters}{bold}{OML}{ntxmi}{b}{it}
  \re@DeclareMathSymbol{\leftharpoonup}{\mathrel}{ntxletters}{"28}
  \re@DeclareMathSymbol{\leftharpoondown}{\mathrel}{ntxletters}{"29}
  \re@DeclareMathSymbol{\rightharpoonup}{\mathrel}{ntxletters}{"2A}
  \re@DeclareMathSymbol{\rightharpoondown}{\mathrel}{ntxletters}{"2B}
  \re@DeclareMathSymbol{\triangleleft}{\mathbin}{ntxletters}{"2F}
  \re@DeclareMathSymbol{\triangleright}{\mathbin}{ntxletters}{"2E}
  \re@DeclareMathSymbol{\partial}{\mathord}{ntxletters}{"40}
  \re@DeclareMathSymbol{\flat}{\mathord}{ntxletters}{"5B}
  \re@DeclareMathSymbol{\natural}{\mathord}{ntxletters}{"5C}
  \re@DeclareMathSymbol{\star}{\mathbin}{ntxletters}{"3F}
  \re@DeclareMathSymbol{\smile}{\mathrel}{ntxletters}{"5E}
  \re@DeclareMathSymbol{\frown}{\mathrel}{ntxletters}{"5F}
  \re@DeclareMathSymbol{\sharp}{\mathord}{ntxletters}{"5D}
  \re@DeclareMathAccent{\vec}{\mathord}{ntxletters}{"7E}
\makeatother

\renewcommand{\epsilon}{\varepsilon}

\DeclareFontFamily{U}{EBSUB}{}
\DeclareFontShape{U}{EBSUB}{m}{it}{<-> EBGaramond-Italic-tlf-lgr}{}
\DeclareFontFamily{U}{EBSUB}{}
\DeclareSymbolFont{EBSUB}{U}{EBSUB}{m}{it}

\DeclareMathSymbol{\mu}{\mathord}{EBSUB}{`m}
\DeclareMathSymbol{\varDelta}{\mathord}{EBSUB}{`D}
\DeclareMathSymbol{\varOmega}{\mathord}{EBSUB}{`W}

 \hypersetup{
    colorlinks,
    linkcolor={red!50!black},
    citecolor={blue!50!black},
    urlcolor={blue!80!black}
}

\usepackage{stackrel}
\usepackage[all]{xy}

\usepackage{bm}

\usepackage{indentfirst}

\usepackage{tocloft}

\title{\bf $\bm{n}$-Valued Groups, Kronecker Sums, \\ and Wendt's Matrices}

\date{}

\author{Victor Buchstaber and Mikhail Kornev}

\begin{document}

\maketitle

\begin{abstract}
The article presents results on the well-known problem concerning the structure of integer polynomials $p_n(z; x, y)$, which define multiplication laws in $n$-valued groups $\mathbb{G}_n$ over the field of complex numbers $\mathbb{C}$. We show that the $n$-valued multiplication in the group $\mathbb{G}_n$ is realized in terms of the eigenvalues of the Kronecker sum of companion Frobenius matrices for polynomials of the form $t^n - x$ in the variable $t$. The notion of a Wendt $(x, y, z)$-matrix is introduced. When $x = (-1)^n$, $y = z = 1$, one recovers the classical Wendt matrix, whose determinant is used in number theory in connection with Fermat’s Last Theorem. It is shown that for each positive integer $n$, the polynomial $p_n$ is given by the determinant of a Wendt $(x, y, z)$-matrix. Iterations of the $n$-valued multiplication in the group $\mathbb{G}_n$ lead to polynomials $p_n(z; x_1, \dots, x_m)$. We prove the irreducibility of the polynomial $p_n(z; x_1, \dots, x_m)$ over various fields. For each $n$, we introduce the notion of classes of symmetric $n$-algebraic $n$-valued groups. The group $\mathbb{G}_n$ belongs to one of these classes. For $n = 2, 3$, a description of the universal objects of these classes is obtained.
\end{abstract}

\tableofcontents

\section{Introduction}\label{sec:intro}\label{intro}

In 1971, V.\,M.~Buchstaber and S.\,P.~Novikov proposed a construction motivated by the theory of characteristic classes \cite{Buchstaber_Novikov}. This construction describes a multiplication in which the product of any pair of elements is a multiset of $n$ points. An axiomatic definition of $n$-valued groups and the first results of their algebraic theory were obtained in a subsequent series of works by V.\,M.~Buchstaber. Currently, the theory of $n$-valued (formal, finite, discrete, topological, and algebro-geometric) groups and their applications in various areas of mathematics and mathematical physics are being developed by a number of authors; see \cite{Buchstaber75, Kholodov81, Buchstaber90, Buchstaber_Veselov, Vershik, BuchRees, Buchstaber, BuchVesGaif, BuchVesnin, Kontsevich_type_polynomials} and the references therein. A comparative analysis of the theories of 1-valued and $n$-valued groups can be found in \cite{Borovik}.

Let us recall the construction of the $n$-valued group over the field of complex numbers $\Cc$ from \cite{Buchstaber_Novikov} (see Section \ref{Basic_concepts_of_n_val} for the basic concepts of the theory of $n$-valed groups). This $n$-valued group will be denoted by $\G_n$ throughout this work. To each pair of complex numbers $x$ and $y$ corresponds the multiset
\begin{equation}\label{operation}
x\ast y = [(\sqrt[n]{x} + \epsilon^r\sqrt[n]{y})^n\mid r = 1, \dots, n]
\end{equation}
of $n$ complex numbers, where $\epsilon$ is a (some) primitive $n$th root of unity, and $\sqrt[n]{\cdot}$ denotes a fixed branch of the root. It is easy to see that the definition of the operation $\ast$ does not depend on the choice of the branch of the root $\sqrt[n]{\cdot}$. Here and below, we will use the principal branch for which $\sqrt[n]{1} = 1$, unless otherwise stated.

It is easy to verify that the operation $\ast$ with unit $0$ and inverse $\inv(x) = (-1)^nx$ defines an $n$-valued associative multiplication (addition) on the set of complex numbers $\Cc$ (see Definition \ref{nval_group_def}). Therefore, one can speak of an $n$-valued group over $\Cc$.

The construction of the group $\G_n$ serves as an example of a general construction of an $n$-valued coset group (see Proposition \ref{coset_theorem}), which is defined by a pair $(G, H)$, where $H$ is a subgroup of order $n$ of the automorphism group of $G$. Coset groups by no means exhaust all $n$-valued groups. In \cite{Buchstaber}, for each positive integer $k$, a $(2k + 1)$-valued group is constructed on a three-element set. In the paper \cite{Ponomarenko24}, it is proved that this group is coset if and only if $4k + 3 = p^s$ for some prime $p$ and positive integer $s$. The same paper discusses the connection between $n$-valued coset groups on a three-element set and the classification of all finite groups of rank 3 obtained in \cite{Liebeck}. In \cite{Vershik}, a connection is indicated between the problem of constructing non-coset $n$-valued groups and the classical problem of Schur \cite{Wielandt}.

An $n$-valued group is called cyclic if it is generated by some element. Unlike ordinary groups, the problem of classifying cyclic $n$-valued groups has turned out to be difficult and remains open. In \cite{Buchstaber}, a broad class of $n$-valued cyclic groups was constructed based on Burnside’s theorem \cite[Chapter XV, Theorem IV]{Burnside55}: for any finite group $G$ possessing an irreducible faithful unitary representation, each of its irreducible representations $\psi$ occurs in the decomposition of the representation $\rho^{\otimes k}$ into a sum of irreducibles for some $k = k(\psi)$. The classification of all finite groups possessing faithful irreducible complex representations was obtained in \cite{Gaschtz54}. It is also noted that in \cite{BuchVesnin}, a wide class of cyclic $n$-valued groups was constructed based on combinatorial group theory. Cyclic $n$-valued groups are important in the theory of multivalued integrable discrete-time dynamics \cite{Veselov91, Buchstaber_Veselov}. In \cite{Buchstaber_Veselov}, the concept of a growth function was proposed as a characteristic of the integrability of such dynamics, and the Euler–Chasles correspondence was described in terms of a representation of a two-valued coset group. In recent paper \cite{Kornev}, the growth function for a class of cyclic coset two-valued groups
$$\mathbb{Z}/m\ast\mathbb{Z}/m = \langle a, b\ | \ a^m = b^m = 1 \rangle$$
and $n$-valued groups
$$\left (\mathbb{Z}/2\right )^{\ast n}= \langle a_1, \dots, a_n \ | \ a_1^2=\dots=a_n^2=1 \rangle$$
for automorphisms cyclically permuting the generators was computed.

In \cite{BuchRees}, the following family of polynomials was introduced:
$$p_n(z; x,y) = \prod\limits_{w\in \inv(x)\ast \inv(y)} (z - w),$$
where $w$ ranges over the $n$-multiset $\inv(x)\ast \inv(y)$. It was noted in the same paper that the polynomials $p_n(z; x, y)$ are integer, homogeneous, symmetric of degree $n$. They are explicitly given in \cite{Buchstaber} for $n= 1, \dots,7$ as integer polynomials in the elementary symmetric polynomials $\sigma_1, \sigma_2$, and $\sigma_3$, and questions were posed regarding the structure of the coefficients of these polynomials depending on $n$.

This motivates the notion of symmetric $n$-algebraic $n$-valued groups (see Definitions \ref{n-algebraic_n-valued} and \ref{symmetric_n-algebraic_n-valued}). For each $n$, a notion of classes of symmetric $n$-algebraic $n$-valued groups is introduced. For $n = 2, 3$, a description of the universal objects of these classes is obtained (Theorems \ref{universal_2-valued_group} and \ref{universal_3-valued_group}). The group $\G_n$ belongs to one of these classes. These results constitute the content of Section \ref{symmetric_n-algebraic_n-valued_groups}.

The groups $\G_n$ and the polynomials $p_n(z; x, y)$ arise and are currently being studied in various areas of mathematics and mathematical physics. In recent paper \cite{Chirkov}, the growth function of the $n$-valued dynamics of the group $\G_n$ was investigated. In another recent work \cite{Gaiur}, a connection of the polynomials $p_n(z; x, y)$ with multiplication kernels in the sense of Kontsevich, $n$-Bessel kernels, and potentials in the Landau–Ginzburg model was described. In \cite{Gaiur}, a link was established between the symmetric 2-algebraic 2-valued group from \cite{Buchstaber90} and Kontsevich polynomials introduced in the framework of the Langlands program.

This work shows that $\G_n$ is realized via the roots of the characteristic polynomial of the variable $z = t^n$ of the Kronecker sum (see Definition \ref{Kronecker_sum}) of two Frobenius companion matrices (\ref{Frobenius_matrix}) for the polynomials $t^n-x$ and $t^n-y$. Iterations of the $n$-valued multiplication lead to homogeneous symmetric polynomials $p_n(z; x_1, \dots, x_m)$, see Theorem \ref{thm1}. This is the content of Section \ref{G_n_and_Kronecker}. In this regard, we recall the necessary constructions from linear algebra in Section \ref{Kronecker_sums_and_companions}. In the same section, we provide a direct and simple proof of Proposition \ref{C_g_circ_C_f}, which first appeared in \cite[Theorem 1]{Carrillo}. This result concerns the realization of polynomial substitution into a polynomial via a certain block companion Frobenius matrix. Our proof relies on a straightforward computation of the determinant of a block matrix from \cite[Lemma 2.1]{Melman}.

In Theorem \ref{xyz_Wendt_det}, it is shown that each polynomial $p_n(z; x, y)$ is represented as the determinant of an $n \times n$ matrix (\ref{xyz_Wendt_det}), which belongs to a class known in the literature as $y$-circulant matrices \cite[p. 211]{Babenko}. Moreover, for each $n$, when $x = (-1)^n$, $y = z = 1$, this determinant $p(z; x, y)$ becomes the determinant $\det W_n$ of Wendt's matrix $W_n$, which has important applications in number theory in the context of Fermat’s Last Theorem (see, for example, \cite[Chapter IV, {\bf (2b)}]{Fermat_for_amateurs}). This allows us to introduce a generalization — the Wendt $(x, y, z)$-matrices. Then Theorem \ref{xyz_Wendt_det} can be equivalently reformulated as the statement that the polynomials $p_n(z; x, y)$, which express $n$-valued multiplication in the group $\G_n$, are determinants of a Wendt $(x, y, z)$-matrix (Theorem \hyperref[xyz_Wendt_and_p_n]{4’}). This result is extended in Theorem \ref{p_n_via_A^n} to the polynomials $p_n(z; x_1\dots x_m)$. This leads to a range of questions related to the number-theoretic properties of the coefficients of the polynomials $p_n(z; x, y)$, generalizing the corresponding one for $W_n$. In Proposition \ref{div_n^4}, using Wolstenholme’s theorem, it is shown that the coefficients of the polynomial $p_n(z; x, y) - (x + y + z)^n$ are divisible by $n^4$ for every prime $n \geqslant 5$. One of the motivations for this work was the experimental observation that arbitrarily large prime numbers appear in the prime factorizations of the coefficients of the polynomials $p_n(z; x, y)$ when expressed in the basis of elementary symmetric polynomials. These results comprise the content of Section \ref{G_n_and_Wendt}.

The polynomials $p_n(z; x_1\dots x_m)$ are naturally connected with several classical and modern problems in number theory, including Waring’s problem. Theorem \ref{trans_extension} states that the minimal polynomial of the element $\theta = {x_1}^{1/n_1} + \ldots + {x_m}^{1/n_m}$ of the extension $\kk(x_1\dots x_m) \subset \kk(\sqrt[n]{x_1}\dots \sqrt[n]{x_m})$ with transcendental $x_1\dots x_m$ over various fields $\kk$ is the characteristic polynomial of the Frobenius sum of the companion matrices of the polynomials $t^{n_j} - x_j$ for $j = 1\dots m$. As a consequence, the polynomials $p_n(z; x_1\dots x_m)$ are irreducible over $\Cc$ and over some other fields. From Theorem \ref{irreducible}, it follows that for suitable integer values of $x_1\dots x_m$, the polynomials $p_n(z; x_1\dots x_m)$ in the single variable $z$ are irreducible over $\Q$. These results follow from the remarkable works \cite{Albu, Ursell}. This is the content of Section \ref{nonderogatory}.

In \cite[Proposition 1]{Helou}, it was noted that the determinant of Wendt's matrix can be represented as the discriminant of a certain univariate polynomial. We generalize this fact and show that the determinant of a Wendt $(x, y, z)$-matrix can be represented as the discriminant of a certain univariate polynomial with parameters $x, y, z$; see Theorem \ref{discriminant_two_dim}. Remarkably, this parametric polynomial already appeared in \cite[Theorem 2.3]{Gaiur} in a problem involving the polynomials $p_n(z; x, y)$, but in a different context unrelated to the determinant of Wendt's matrix. In Theorem \ref{multi_disc}, this result is partially extended to the polynomials $p_n(z; x_1\dots x_m)$. This is the content of Section \ref{p_n_and_disc_varieties}.

In survey \cite{Buchstaber}, for each finite $G$, a construction is introduced of an $n$-valued commutative group on the set $\{\rho/\dim(\rho)\}$ of normalized irreducible unitary representations of the group $G$. We denote this group by $\G^{\rep}_{\otimes}$. Cyclic $n$-valued groups constructed based on Burnside’s theorem (see above) have the form $\G^{\rep}_{\otimes}$. The group $\G^{\rep}_{\otimes}$ is an example of an $n$-valued group that belongs to the class of $n$-valued groups which, according to \cite{Vershik}, can be identified with the class of combinatorial algebras playing an important role in algebraic combinatorics. The construction of the group $\G^{\rep}_{\otimes}$ is related to the well-known problem of the coefficients in the decomposition of tensor products of representations into irreducibles \cite{Littlewood06}. If one considers the Kronecker sum instead of the tensor product, then Theorem \ref{Kronecker_sum_group} states that on the set of all normalized finite-dimensional representations $\{\psi/\dim\psi\}$ of a finite group $G$, there arises a structure of a 1-valued commutative group $\G^{\rep}_{\boxplus}$. This is the content of Section \ref{characters}.

For $n$-valued groups, there is a natural definition of the Cartesian product (see Definition \ref{cartesian}). Considering the left action of the group $\G^m_n$ on the group $\G_n$, we arrive at a dynamical interpretation of the polynomials $p_n(z; x_1\dots x_m)$ in Theorem \ref{action}. This is the content of Section \ref{G^m_n}.

\section{Basic Concepts of the Theory of $n$-Valued Groups}\label{Basic_concepts_of_n_val}

In this section, we give some definitions and examples, following \cite{Buchstaber}.

\begin{defi}\label{nval_group_def}
An {\it $n$-valued multiplication} on a set $X$ is a map
$$\mu: X\times X \to \Sym^n(X)\,:\,\mu(x,y)=x*y,$$
where $\Sym^n(X):=X^n/\Sigma_n$ is the $n$-th symmetric power of the set $X$ (i.e., the set consisting of multisets $[z_1\dots z_n]$), such that the following conditions are satisfied:

\begin{itemize}

\item {\it Associativity.} The $n^2$-multisets
\begin{gather*}
[x\ast w\mid w\in y \ast z],\\
[w\ast z\mid w\in x\ast y]
\end{gather*}
coincide.

\item {\it Identity.} There exists an element $e\in X$ such that $$e\ast x = x\ast e = [x, x, \dots, x]$$ for all $x\in X$.

\item {\it Inverse element.} There is a map $\mathrm{inv}\colon X\to X$ such that
$$e\in \mathrm{inv}(x)* x,\quad e \in x*\mathrm{inv}(x)$$
for every $x\in X$.

\end{itemize}
\end{defi}

\begin{defi}
If a set $X$ has an $n$-valued multiplication, then $X$ is called an {\it $n$-valued group}.
\end{defi}

\begin{example}
Every 1-valued group $G$ is an ordinary group.
\end{example}

\begin{example}\label{G_n_defi}
The set $\Cc$ of complex numbers carries a structure of an $n$-valued group $\G_n$. Consider the operation $\ast: \Cc\times\Cc\to \Sym^n(\Cc)$ that assigns to each pair of complex numbers $x$ and $y$ the multiset
\begin{equation}\label{operation2}
x\ast y = [(\sqrt[n]{x} + \epsilon^r\sqrt[n]{y})^n\mid r = 1\dots n]
\end{equation}
of $n$ complex numbers, where $\epsilon$ is a primitive $n$th root of unity, and $\sqrt[n]{\cdot}$ denotes a fixed root branch. The operation $\ast$, the identity element $0$, and the inverse $\inv(x) = (-1)^n x$ endow the set of complex numbers $\Cc$ with the structure of an $n$-valued group.
\end{example}

In Section \ref{symmetric_n-algebraic_n-valued_groups}, we consider examples of $n$-algebraic $n$-valued groups. We refer the reader to the list of cited references for many other examples.

\begin{defi}
An $n$-valued group on a set $X$ is called {\it commutative} if $x_1\ast x_2 = x_2\ast x_1$ for all $x_1, x_2\in X$.
\end{defi}

\begin{defi}
A map $f: X\to Y$ between $n$-valued groups is called a {\it homomorphism} if

\begin{itemize}

\item $f(e_X)=e_Y$.

\item $f(\mathrm{inv}_X(x))=\mathrm{inv}_Y(f(x))$ for any $x\in X$.

\item $\mu_{Y}(f(x),f(y))=(f)^n\mu_X(x,y)$ for all $x, y\in X$.

\end{itemize}

That is, the following diagram is commutative:

\[
\xymatrix{
X\times X \ar[r]^{\mu_X} \ar[d]_{f\times f} & \Sym^n(X) \ar[d]^{\Sym(f)} \\
Y\times Y \ar[r]       & \Sym^n(Y) }
\]

A bijective homomorphism of $n$-valued groups is called an {\it $n$-isomorphism}.

\end{defi}

Thus, the class of $n$-valued groups forms a category $\nGrp$.

There is a method for constructing $n$-valued groups \cite{Buchstaber}, in which for a given (1-valued) group $G$ with multiplication $\mu_0$, identity $e_G$, and inverse $\inv_G(u) = u^{-1}$, and for a given finite subgroup $H$ of the automorphism group $\Aut(G)$, one considers the set of orbits $X := G/H$ of $G$ under the action of $H$ with projection $\pi: G\to X$ and an $n$-valued multiplication
$$\mu: X\times X\to \Sym^n(X),$$
given by the formula
\begin{equation}\label{coset}
\mu(x, y) = [ \pi(\mu_0(u, h(v))) \ | \ h\in H ]
\end{equation}
where $u\in\pi^{-1}(x)$ and $v\in\pi^{-1}(y)$.

Note that non-isomorphic pairs $(G_1, H_1)$ and $(G_2, H_2)$ may produce isomorphic $n$-valued groups \cite[Proposition 3.2]{BuchRees}.

\begin{prop}[(Coset construction, \cite{Buchstaber})]\label{coset_theorem}
The multiplication $\mu$ given by {\normalfont (\ref{coset})} defines an $n$-valued multiplication on the orbit space $X = G/H$ with identity $e_x = \pi(e_G)$ and inverse $\inv_X(x) = \pi(\inv_G(u))$, where $u\in\pi^{-1}(x)$.
\end{prop}

\begin{defi}\label{coset_defi}
For each group $G$ and a subgroup $H$ of order $n$ in the group $\Aut(G)$, the $n$-valued group on the orbit set $G/H$ from Proposition \ref{coset_theorem} is called a {\it coset group} or the {\it coset construction}.
\end{defi}

Note that the axioms of $n$-valued multiplication do not imply the uniqueness of the inverse element map. However, in the case of the coset construction, the inverse is indeed unique.

\begin{defi}
An {\it $n$-valued dynamical system} $T$ on a space $X$ is a map $T: X\to \Sym^n(X)$.
\end{defi}

One may think of $X$ as a state space. Then an $n$-valued dynamics $T$ assigns to each element $x$ the possible states $T(x) = [x_1\dots x_n]$ at time $t+1$ as a function of the state $x$ at time $t$.

\begin{defi}\label{group_action}
An {\it action of an $n$-valued group $G$ on a space} $X$ is a map
$$\circ: G\times X\to \Sym^n(X),\ g\circ x = [x_1\dots x_n],$$
such that

\begin{itemize}

\item For all $g_1, g_2\in G$ and all $x\in X$, the following $n^2$-multisets coincide:
$$g_1\ast(g_2\circ x) = [g_1\circ x_1\dots g_1\circ x_n]$$
and
$$(g_1\ast g_2)\circ x = [h_1\circ x\dots h_n\circ x],$$
where $g_2\circ x = [x_1\dots x_n]$ and $g_1\ast g_2 = [h_1\dots h_n]$.

\item $e\ast x = [x\dots x]$ for all $x\in X$.

\end{itemize}

\end{defi}

\begin{example}
Every $n$-valued group $G$ acts on itself by left translations, just like single-valued groups. As shown in the recent work \cite{Posadskiy}, not every 2-valued dynamical system on a set $X$ can be realized by the action of some $n$-valued group on $X$.
\end{example}

In \cite{Buchstaber_Veselov}, a condition for integrability of an $n$-valued dynamical system on a space $X$ is formulated in terms of a representation of some $n$-valued group $G$ on $X$. In \cite{Yagodovskii}, a sufficient condition for the integrability of multivalued dynamics using a multivalued group is obtained.

\section{Symmetric $n$-Algebraic $n$-Valued Groups}\label{symmetric_n-algebraic_n-valued_groups}

Consider the operation $\ast$ that assigns to elements $x$ and $y$ in the set of complex numbers $\Cc$, such that $F_0(x, y)\neq 0$, the multiset $x\ast y$ of the roots of the polynomial in the variable $z$
\begin{equation}\label{polynom_P}
P(z; x, y) = \sum\limits_{k = 0}^nF_{n-k}(x, y)z^{k},
\end{equation}
where $F_{n-k}(x,y)\in\Cc[x, y]$ and
\begin{equation}\label{F_0(0, 0)=1}
F_0(0, 0) = 1.
\end{equation}
 In this context, let us formulate the axioms from Definition \ref{nval_group_def}.

From the neutral element axiom with $x = 0$, it follows that
\begin{equation}\label{G(P)_neutral_element}
P(z; 0, y) = F_0(0, y)(z - y)^n.
\end{equation}

The inverse element axiom asserts the existence of an algebraic map $\inv: \Cc\to \Cc$, such that for every $x\in\Cc$
\begin{equation}\label{G(P)_inverse_element}
F_n(x, \inv(x)) = F_n(\inv(x), x) = 0.
\end{equation}

Let $z_j = z_j(x, y)$, $j = 1\dots n$, denote the roots of the polynomial $P$. Then the associativity condition is written as the identity of two polynomials in the variable $t$:
\begin{equation}\label{assoc_condition}
\begin{aligned}
\prod\limits_{j, k = 1}^n (t - z_j(z_k(x, y), z)) = \prod\limits_{j, k = 1}^n (t - z_j(x, z_k(y, z))).
\end{aligned}
\end{equation}
Note that the coefficients of the polynomials in $t$ on both sides of equation (\ref{assoc_condition}) are polynomials in $x$, $y$, and $z$.

\begin{defi}
An {\it $n$-valued algebraic group} is an algebraic variety $X$ with an $n$-valued multiplication defined on a dense subset $Y\subset X\times X$ in the Zariski topology.
\end{defi}

\begin{remark*}
Our definition of a 1-algebraic 1-valued group uses a partially defined operation and does not, therefore, generally coincide with the classical notion of an algebraic group.
\end{remark*}

The definitions of commutativity and homomorphism are obvious for algebraic $n$-valued groups.

\begin{defi}\label{n-algebraic_n-valued}
An {\it $n$-algebraic $n$-valued group} on $\Cc$ is the set $\Cc$ with a partially defined $n$-valued multiplication
$$\xymatrix@R=5pt{\ast: (\Cc\times\Cc)\backslash \{(x, y)\mid F_0(x, y) = 0\}\ar[r] & \Sym^n(\Cc) \\
(x, y) \ar@{|->}[r] & [z_1\dots z_n]}$$
where $z_j$ are the roots of polynomial (\ref{polynom_P}), such that each of the variables $x$ and $y$ appears in it with degree at most $n$, and the polynomial satisfies conditions (\ref{F_0(0, 0)=1}), (\ref{G(P)_neutral_element}), (\ref{G(P)_inverse_element}), and (\ref{assoc_condition}). We denote the corresponding group by $\G(P)$.
\end{defi}

It is easy to see that an $n$-algebraic $n$-valued group is an $n$-valued group if and only if $F_0(x, y)\equiv 1$ in (\ref{polynom_P}).

\begin{example}
The polynomial $z - x - y$ defines a 1-algebraic 1-valued group on $\Cc$ with the usual addition as the operation.
\end{example}

\begin{example}\label{alpha_x_y_z}
Let
$$P(z; x, y) = -x - y + z + \alpha xyz,$$
where $\alpha\in\Cc$. Then we have a 1-algebraic 1-valued commutative group on $\Cc$ with the operation
$$x\ast y = \frac{x + y}{1 + \alpha xy}.$$
This algebraic abelian group is isomorphic to the algebraic 1-valued group of matrices of the form
$$\left(\begin{array}{cc}
1 & x\\
\alpha x & 1
\end{array}\right)$$ modulo multiplication by nonzero complex numbers, since the following identity holds:
$$\left(\begin{array}{cc}
1 & x\\
\alpha x & 1
\end{array}\right)\cdot \left(\begin{array}{cc}
1 & y\\
\alpha y & 1
\end{array}\right) = (1 + \alpha xy) \left(\begin{array}{cc}
1 & \frac{x + y}{1 + \alpha xy}\\
\alpha \frac{x + y}{1 + \alpha xy} & 1
\end{array}\right).$$
\end{example}

\begin{defi}\label{symmetric_n-algebraic_n-valued}
An $n$-algebraic $n$-valued group is called {\it symmetric} if the polynomial $$P(z; (-1)^nx, (-1)^ny)$$ is symmetric in the variables $x$, $y$, and $z$.
\end{defi}

\begin{prop}
Let $P(z; x, y)$ define the multiplication in a symmetric $n$-algebraic $n$-valued group. Then the following statements hold:

\begin{enumerate}[(i)]

\item $F_0(0, y) \equiv 1$,

\item $F_n(x, y) = (x - y)^n$,

\item $\inv(x) = (-1)^nx$.

\end{enumerate}
\end{prop}

\begin{proof}
Note that from conditions (\ref{F_0(0, 0)=1}) and (\ref{G(P)_neutral_element}) it follows that $F_0(0, y) \equiv 1$. The rest is obvious.
\end{proof}

Note that the map $\inv$ is uniquely defined in this case.

Clearly, every symmetric $n$-algebraic $n$-valued group is commutative.

\begin{example}\label{G(p_n)}
For each natural number $n$, the group $\G_n$ from Example \ref{G_n_defi} is a symmetric $n$-algebraic $n$-valued group $\G(p_n)$ corresponding to the integer polynomial
$$p_n(z; (-1)^nx, (-1)^ny) = \prod\limits_{r = 1}^n (z - (\sqrt[n]{x} + \epsilon^r\sqrt[n]{y})^n),$$
where $\epsilon$ is a primitive $n$th root of unity. These polynomials were first introduced in \cite{Buchstaber_Novikov} in the context of the theory of $n$-valued groups. The problem of describing the structure of the coefficients of the polynomials $p_n$ was posed in \cite{Buchstaber}, where also the explicit form of $p_n(z; x, y)$ was given for $1\leqslant n\leqslant 7$. Results in this direction are obtained further in the present work.
\end{example}

For each natural number $n$, define the polynomial
\begin{equation}\label{P_universal}
P_n(z; x, y) = \sum\limits_{k=0}^n\left(\sum\limits_{i,j = 0}^n a_{ijk}x^iy^j\right) z^k
\end{equation}
with indeterminate coefficients and the following grading:
$$\deg a_{ijk}  = 2(n - i - j -k),$$
$$\deg x = \deg y = \deg z = 2.$$

The homogeneous summand
$$\sum\limits_{i+j+k = n}a_{ijk}x^iy^jz^k$$
in the polynomial (\ref{P_universal}) will be called the {\it homogeneous $n$-component $P_n^{\h}$ of the polynomial $P_n$}.

It is easy to verify that if (\ref{P_universal}) defines a multiplication law in an $n$-valued algebraic group, then its homogeneous $n$-component also defines a certain $n$-valued multiplication law.

In the problem of classifying algebraic $n$-valued groups, an important initial task is the classification of the polynomials $P_n^{\h}$. Results in this direction are presented in the works \cite{Buchstaber75, Kholodov81, Kholodov84}.

\begin{defi}
The {\it universal symmetric $n$-algebraic $n$-valued group over $\Cc$ with fixed homogeneous $n$-component $P_n^{\h}$} is defined by the polynomial (\ref{P_universal}) with coefficients in the ring
$$\Cc[a_{ijk}\mid \deg a_{ijk} < 0, i, j, k \in {1\dots n}]/\mathcal{I},$$
where $\mathcal{I}$ is the ideal defined by the axioms of a symmetric $n$-algebraic $n$-valued group.
\end{defi}

For $n = 2$, from the degree count and algebraic independence of the elementary symmetric polynomials $\sigma_1, \sigma_2$, and $\sigma_3$ in the graded variables $x, y$, and $z$, it follows that each polynomial (\ref{polynom_P}) can be written in the form
\begin{equation}\label{P_2_universal}
P_2(z; -x, -y) = \sigma_1^2 - 4\sigma_2 + k_2\sigma_3 + k_8\sigma_3^2 + k_4\sigma_1\sigma_3 + k_6\sigma_2\sigma_3
\end{equation}
for indeterminate graded complex coefficients $k_j$ with $\deg k_j = -j < 0$. As we will show below, in this case the homogeneous component $P_2^{\h}$ is uniquely determined and coincides with the polynomial $p_2(z; x, y)$ in the notation of Example \ref{G(p_n)}.

With respect to this grading, the polynomial $P_2$ becomes homogeneous in the variables $x$, $y$, $z$, and $k_j$. After expanding, we obtain
\begin{align}
P_2(z; -x, -y) &= z^2 \left(k_8 x^2 y^2+k_4 x y+k_6 x^2 y+k_6 x y^2+1\right) \notag \\
&+ z \left(k_2 x y+k_4x^2 y+k_4 x y^2+k_6 x^2 y^2-2 x-2 y\right) \notag \\
&+ (x - y)^2\notag \\
&=: a(x, y)z^2 + b(x, y)z + c(x, y).\label{P_2_a_b_c}
\end{align}

The associativity equation for the case $n = 2$ takes the form:
\begin{equation}\label{Q(t)}
\prod\limits_{i,j = 1}^2 (t - z_{ij}) = \prod\limits_{i, j =1}^2 (t - \widetilde{z_{ij}}),
\end{equation}
where $z_{ij} = z_i(z_j(x, y), z)$ and $\widetilde{z_{ij}} = z_i(x, z_j(y ,z))$.

Equating the coefficients of $t^3$ in (\ref{Q(t)}) and using Vieta’s formulas for the roots of the polynomial $P_2$ with respect to the variable $z$, we obtain the condition
\begin{equation}\label{fractions}
\frac{b_1}{a_1} + \frac{b_2}{a_2} = \frac{\widetilde{b_1}}{\widetilde{a_1}} + \frac{\widetilde{b_2}}{\widetilde{a_2}},
\end{equation}
where $a_j = a(z_j(x, y), z)$, $b_j = b(z_j(x, y), z)$, $c_j = c(z_j(x, y), z)$, $\widetilde{a_j} = a(x, z_j(y, z))$, $\widetilde{b_j} = b(x, z_j(y, z))$, $\widetilde{c_j} = c(x, z_j(y, z))$, $j = 1, 2$. Reducing the fractions in (\ref{fractions}) to a common denominator, we get the equality
\begin{equation}\label{A_1B_2=A_2B_1}
A_1B_2 = A_2B_1,
\end{equation}
where
$$
\begin{aligned}
A_1 &= b_1a_2 + b_2a_1, \
A_2 &= \widetilde{b_1}\widetilde{a_2} + \widetilde{b_2}\widetilde{a_1},  \
B_1 &= a_1a_2, \
B_2 &= \widetilde{a_1}\widetilde{a_2}.
\end{aligned}
$$
Since the pairs of roots $(z_1(x, y), z_2(x, y))$ and $(z_1(y, z), z_2(y, z))$ enter the polynomials $A_1B_2$ and $A_2B_1$ symmetrically, we again apply Vieta’s formulas to rewrite the equality (\ref{A_1B_2=A_2B_1}) in terms of the coefficients of the polynomial (\ref{P_2_a_b_c}).

As calculations using the computer algebra system {\tt Wolfram Mathematica} show, the coefficient of $t^3$ in the polynomial $Q = A_1B_2 - A_2B_1$ is divisible by $4k_8 - k_4^2 + k_6k_2$, and the coefficient of the monomial $x^3yz$ is equal to $-6(4k_8 - k_4^2 + k_6k_2)$. This means that the parameters of the universal symmetric 2-algebraic 2-valued group satisfy the relation
\begin{equation}\label{f^2-gd=0}
4k_8 = k_4^2 - k_6k_2.
\end{equation}

The condition for matching the coefficients at $t^2$ in (\ref{Q(t)}) becomes the equality
\begin{equation}
\frac{c_1}{a_1} + \frac{c_2}{a_2} + \frac{b_1b_2}{a_1a_2} = \frac{\widetilde{c_1}}{\widetilde{a_1}} + \frac{\widetilde{c_2}}{\widetilde{a_2}} + \frac{\widetilde{b_1}\widetilde{b_2}}{\widetilde{a_1}\widetilde{a_2}}.
\end{equation}

For the coefficient at $t$, we get
$$
\frac{1}{a_1a_2}(b_2c_1 + b_1c_2) = \frac{1}{\widetilde{a_1}\widetilde{a_2}}(\widetilde{b_2}\widetilde{c_1} + \widetilde{b_1}\widetilde{c_2}).
$$

At $t^0$ we have
$$
\frac{c_1c_2}{a_1a_2} = \frac{\widetilde{c_1}\widetilde{c_2}}{\widetilde{a_1}\widetilde{a_2}}.
$$

Analogous computations in {\tt Wolfram Mathematica} show that the differences in coefficients at each $t^j$, $j \in {2, 1, 0}$ are divisible by $4k_8 - k_4^2 + k_6k_2$.

Thus, we obtain the following classification result:

\begin{theorem}\label{universal_2-valued_group}
The universal symmetric 2-algebraic 2-valued group (\ref{P_2_universal}) over $\Cc$ is defined over the coefficient ring
$$\Cc[k_2, k_4, k_6, k_8]/(4k_8 - k_4^2 + k_6k_2)\cong \Cc[k_2, k_4, k_6],$$
i.e., the universal family of polynomials $P_2$ has the form:
\begin{equation}\label{family}
P_2(z; x, y) = \sigma_1^2 - 4\sigma_2 + 2k_2\sigma_3 + (k_4^2 - k_6k_2)\sigma_3^2 + 2k_4\sigma_1\sigma_3 + 2k_6\sigma_2\sigma_3.
\end{equation}
\end{theorem}

Thus, for $n = 2$ we obtain a 3-parameter family of polynomials $P_2$.

In \cite{Buchstaber75}, a classification was obtained for 2-valued formal groups corresponding to polynomials in the variable $z$ of the form
$$z^2 - \Theta_1z + \Theta_2,$$
where $\Theta_1 = \Theta_1(x, y)$ and $\Theta_2 = \Theta_2(x, y)$ are formal power series over some commutative ring $\kk$. In the case of $\kk$ is a $\Q$-algebra, the universal 2-valued formal group is defined over the graded ring $\Q[a_{4n} \mid n\in \mathbb{N}]$.

In \cite{Buchstaber90}, a symmetric 2-algebraic 2-valued group was constructed using the coset method for an involution of order 2 on a nonsingular elliptic curve in Weierstrass form. The resulting multiplication law coincides with (\ref{family}), but under an additional condition ensuring nonsingularity of the elliptic curve.

The nonsingularity condition was lifted in recent work \cite{Buchstaber_Veselov19}.

Note that the works \cite{Buchstaber90, Buchstaber_Veselov19} used the fundamental addition theorem for Weierstrass elliptic functions. Our method uses only the concept of a universal symmetric $n$-algebraic $n$-valued group.

Let us now describe all symmetric 3-algebraic 3-valued groups.

The general form of the polynomials $P_3$ is as follows:
\begin{equation}\label{P_3}
\sigma_1^3 + k_0\sigma_3 + k_6\sigma_3^2 + k_{12}\sigma_3^3 + k_2\sigma_1\sigma_3 + k_4\sigma_2\sigma_3 + \ell_6\sigma_1\sigma_2\sigma_3 + \ell_4\sigma_1^2\sigma_3 + k_8\sigma_2^2\sigma_3 + \ell_8\sigma_3^2\sigma_1 + k_{10}\sigma_3^2\sigma_2,
\end{equation}
where the undetermined coefficients are graded by $\deg(k_i) = -i$, $\deg(\ell_j) = -j$, and each of the variables $x$, $y$, and $z$ has degree 2.

\begin{theorem}\label{universal_3-valued_group}
There are two classes of symmetric 3-algebraic 3-valued groups:

\begin{enumerate}[(i)]

\item If $k_0 = -27$, we have $$P_3(z; -x, -y) = \sigma_1^3 - 27 \sigma_3 + 18 c \sigma_1^2 \sigma_3 - 54 c \sigma_2 \sigma_3 - 27 c^2 \sigma_2^2 \sigma_3 + 81 c^2 \sigma_1 \sigma_3^2.$$ The corresponding homogenious 3-component is $P_3^{\h}(z; -x, -y) = \sigma_1^3 -27\sigma_3$ (the 3-valued group $\G_3$ from Example \ref{G_n_defi}).

\item If $k_0 = 0$, we have
$$P_3(z; -x, -y) = \left(\sigma_1 + \alpha\sigma_3\right)^3,$$
The corresponding homogenious 3-component is $P_3^{\h}(z; -x, -y) = \sigma_1^3$. In this case, the 3-valued group is obtained by the diagonal construction \cite[Lemma 1]{Buchstaber}
$$\xymatrix{G\times G\ar[r]^{\;\;\;\;\mu} & G\ar[rr]^{\Sym^3\circ\diag\;\;\;\;\;\;\;} & & \Sym^3(G)}$$
from the 1-algebraic 1-valued group $G$, constructed by the polynomial $P_1(z; -x, -y) = \sigma_1 + \alpha\sigma_3$ (see Example \ref{alpha_x_y_z}).

\end{enumerate}

\end{theorem}

\begin{proof}

Let us write the polynomial (\ref{P_3}) in terms of the variables $x$, $y$, and $z$:
$$\begin{aligned}P_3 &= \textcolor{blue}{z^3} (k_{12} x^3 y^3+k_{10} x^3 y^2+k_8 x^3 y+k_{10} x^2 y^3+2 k_8 x^2 y^2 \\ &+ k_8 x y^3+\ell_8 x^2 y^2+\ell_6 x^2 y+\ell_6 x y^2+\ell_4 x y+1) \\ &+ \textcolor{blue}{z^2} (k_{10} x^3 y^3+2 k_8 x^3 y^2+2 k_8 x^2 y^3+k_6 x^2 y^2 \\ &+ k_4 x^2 y+k_4 x y^2+k_2 x y+\ell_8 x^3 y^2+\ell_6 x^3 y+\ell_8 x^2 y^3 \\ &+ 3 \ell_6 x^2 y^2+2 \ell_4 x^2 y+\ell_6 x y^3+2 \ell_4 x y^2+3 x+3 y) \\ &+ \textcolor{blue}{z} (k_8 x^3 y^3+k_4 x^2 y^2+k_2 x^2 y+k_2 x y^2+k_0 x y+\ell_6 x^3 y^2+\ell_4 x^3 y+\ell_6 x^2 y^3 \\ &+ 2 \ell_4 x^2 y^2+\ell_4 x y^3+3 x^2+6 x y+3 y^2) \\ &+ (x + y)^3 \\ &=: a(x, y)\textcolor{blue}{z^3} + b(x, y)\textcolor{blue}{z^2} + c(x, y)\textcolor{blue}{z} + d(x, y).\end{aligned}$$

Denote by $Q(t)$ the 8\textit{th}-degree polynomial of the variable $t$ obtained from the difference between the left and right-hand sides of equality (\ref{assoc_condition}). Also, let $Q_i(x, y, z)$ denote its coefficient of $t^i$ for each $i = 0\dots 8$.

From the introduced grading, it follows that each monomial $x^\alpha y^\beta z^\gamma$ in $Q(t)$ is accompanied by a homogeneous polynomial in the indeterminate coefficients $k_i$, $\ell_j$. Suppose that $\deg t = 2$. Then $Q(t)$ is homogeneous in $x$, $y$, $z$, $k_i$, $\ell_j$ of degree 18. 

To determine the conditions on the coefficient $k_0$, it suffices to set all other coefficients $k_i$, $\ell_j$ in $Q(t)$ to zero. A direct computation of the coefficient $Q_5(x, y, z)$ of $t^5$ in the polynomial $Q(t)$ under these assumptions in {\tt Wolfram Mathematica} shows that this coefficient includes the factor $k_0^2 \left(k_0+27\right)$. Hence, $k_0$ can only take two values: $-27$ and $0$.

Setting the corresponding partial derivatives at $x = y = z = 0$ of the polynomial in $x$, $y$, and $z$ of $t^8$ to zero gives the following system of seven algebraic equations in the indeterminate coefficients:
 \begin{equation}\label{big_system}\left\{ \begin{aligned}
-k_0 k_2 \ell_4+k_2 \ell_4+6 k_0 \ell_6+2 k_2 k_4+k_0 k_6-3 k_6 = 0 &\text{ $\;$ at }x^{2}yz,\\
\left(k_0+9\right) k_8 l_6 = 0 &\text{  $\;$ at }x^{6}yz,\\
\left(k_{0}+27\right)k_{8}^{2} =0 &\text{  $\;$ at }x^{7}yz,\\
 \begin{aligned}&9k_{2}^{2}\ell_{6}+3k_{0}k_{2}\ell_{4}^{2}-6k_{2}\ell_{4}^{2}-19k_{4}k_{2}\ell_{4}+9k_{0}k_{2}\ell_{8}-15k_{2}\ell_{8} \\ &-3k_{0}k_{6}\ell_{4}+33k_{6}\ell_{4}+10k_{0}k_{4}\ell_{6}+84k_{4}\ell_{6}-91k_{0}\ell_{4}\ell_{6} \\ &-2k_{4}^{2}k_{2}-42k_{0}k_{8}k_{2}+105k_{8}k_{2}-6k_{4}k_{6}+45k_{0}k_{10}\\ &-45k_{10}+15\ell_{4}\ell_{6} = 0\end{aligned} &\text{  $\;$ at }x^{3}y^{2}z,\\
\text{Continued on next page...}\\
\end{aligned}\right.
\end{equation}

\begin{equation*}\left\{ \begin{aligned}
\begin{aligned}&k_{0}k_{4}\ell_{6}+18k_{4}\ell_{6}-3k_{0}\ell_{4}\ell_{6}+k_{0}k_{2}k_{8} \\ &+21k_{2}k_{8}+3k_{0}k_{10}-9k_{10}+3\ell_{4}\ell_{6} = 0\end{aligned} &\text{  $\;$ at }x^{4}yz,\\
-2k_{0}\ell_{6}^{2}-4k_{0}k_{8}\ell_{4}+k_{0}k_{4}k_{8}+63k_{4}k_{8}-6\ell_{6}^{2} =0 &\text{  $\;$ at }x^{5}yz,\\
\begin{aligned}&2 k_2^2 \ell_6+2 k_2 \ell_4^2-3 k_4 k_2 \ell_4-8 k_2 \ell_8+2 k_6 \ell_4\\ &+8 k_0 k_4 \ell_6+36 k_4 \ell_6-23 k_0 \ell_4 \ell_6-18 k_0 k_8 k_2\\ &+38 k_8 k_2-k_2 k_4^2+36 k_0 k_{10}-18 k_{10}+6 \ell_4 \ell_6 = 0\end{aligned} &\text{  $\;$ at }x^{2}y^{3}z.
\end{aligned}\right.
\end{equation*}

Two more equations follow from the grading property. Let us set $k_{12} = 0$ and take the reduction of the $t^8$ coefficient modulo the ideal generated by the polynomials $x^2R$, $y^2R$, $z^2R$, $xyR$, $yzR$, $xzR$, where $R = xyz(x-z)$. This is allowed since $\deg Q_8 = 2$. Then the conditions of vanishing of the coefficients at $x$ and $y$ yield the following equations:
\begin{equation}\label{small_system}\left\{ \begin{aligned}
\begin{aligned}&-k_{0}k_{4}\ell_{4}-k_{4}\ell_{4}+k_{0}\ell_{4}^{2}-k_{0}k_{2}\ell_{6}-6k_{2}\ell_{6} \\ &-2k_{0}\ell_{8}+3k_{4}^{2}+29k_{0}k_{8}+3k_{8}-2\ell_{4}^{2}+6\ell_{8} = 0,\end{aligned}\\
\begin{aligned}&-7k_{0}\ell_{4}^{2}+7k_{4}k_{0}\ell_{4}+3k_{2}k_{0}\ell_{6}+18k_{0}\ell_{8}-2k_{2}^{2}\ell_{4}\\ &+3k_{4}\ell_{4}+18k_{2}\ell_{6}+6k_{8}k_{0}^{2}-81k_{8}k_{0}-9k_{4}^{2}-k_{2}^{2}k_{4}\\ &+3k_{2}k_{6}-9k_{8}+6\ell_{4}^{2}-18\ell_{8} = 0.\end{aligned}
\end{aligned}\right.\end{equation}

We obtain the following set of solutions to systems (\ref{big_system}) and (\ref{small_system}) for $k_0 = -27$, using $\land$ and $\lor$ to denote conjunction and disjunction respectively:

\begin{enumerate}[\bf 1)]

\item $k_2=0\land k_6=0\land k_8=0\land k_{10}=0\land \left(\ell_4=-\frac{k_4}{3}\lor \ell_4=k_4\right)\land \ell_6=0$ \\ $\land \ell_8=\frac{1}{9} \left(k_4^2-k_4 \ell_4\right)$. 

\item $k_6=k_2 k_4\land k_8=0\land k_{10}=0\land \ell_4=k_4\land \ell_6=0\land \ell_8=0\land k_2k_4(563 k_4-36 k_2^2)\neq 0$. 

\item $k_2=0\land k_4=0\land k_6=0\land k_8=0\land k_{10}=0\land \ell_4=0\land \ell_6=0\land \ell_8=0$. 

\item $k_2=0\land k_4=0\land k_6=0\land k_{10}=0\land k_8\neq 0\land \ell_4=0\land \ell_6=0\land \ell_8=13 k_8$. 

\item $k_2=0\land k_6=0\land k_{10}=0\land k_8\neq 0\land \ell_4=-\frac{k_4}{3}\land \ell_6=0$ $\land \ell_8=-3k_8\land k_4\neq 0$.

\item $k_2=0\land k_6=0\land k_{10}=0\land k_8\neq 0\land \ell_4=-\frac{k_4}{3}\land \ell_6=0\land \ell_8=\frac{1}{9} \left(-k_4 \ell_4+k_4^2+117 k_8\right)\land k_4\neq 0$

\item $k_4=0\land k_6=0\land k_8=0\land k_{10}=0\land \ell_4=0\land \ell_6=0\land \ell_8=0\land k_2\neq 0$. 

\item $k_2 k_4\neq 0\land k_6=k_2 k_4\land k_8=0\land k_6\neq 0\land k_{10}=0\land \ell_4 = k_4\land \ell_6=0\land \ell_8=0$. 

\item $k_4=\frac{36 k_2^2}{563}\land k_2\neq 0\land k_6=\frac{36 k_2^3}{563}\land k_6\neq 0\land k_8=0\land k_{10}=0\land \ell_4=\frac{36 k_2^2}{563}\land \ell_6=0\land \ell_8=0.$

\end{enumerate}

Let us consider case {\bf 1)}. If $\ell_4 = -k_4/3$, then computing the coefficient of $x^4yz$ in the polynomial $Q_7(x, y, z)$ gives the vanishing of the coefficient $k_4$. If instead $\ell_4 = k_4$, computing the coefficient of $x^4yz$ also gives the vanishing of $k_4$. Hence, case {\bf 1)} reduces to case {\bf 3)}.

Let us consider case {\bf 2)}. Computing the coefficients of $x^2yz$ and $x^3yz$ in the polynomial $Q_7(x, y, z)$ yields the system
$$\begin{cases}
k_{2}^{2}-81\ell_{4} & =0,\\
k_{2}\ell_{4} & =0,
\end{cases}$$
from which it follows that $k_2 = \ell_4 = 0$, so this case is eliminated.

Consider {\bf 3)}. Substituting these parameter values into the polynomial $Q_8(t)$, we obtain that $k_{12} = 0$. This case is contained in the modification of the case {\bf 5)} with $c = 0$ (the same argument still holds true), see below.

Let us consider case {\bf 4)}. Computing the coefficient of $x^4yz$ in the polynomial $Q_7(x, y, z)$ gives the vanishing of $k_8$, eliminating this case.

Let us consider case {\bf 5)}.  Computing the coefficient of $x^3y^2z$ in the polynomial $Q_7(t)$ gives the vanishing of $k_4^2 + 108k_8$. Hence, for each $c\in\Cc$, we get \begin{equation}\left\{\begin{aligned}\label{c_case_equations}k_0 &= -27,\\ k_2 &= 0,\\ k_4 &= -54c,\\ k_6 &= 0,\\ k_8 &= -27c^2,\\ k_{10} &= 0,\\ \ell_4 &= 18c,\\ \ell_6 &= 0,\\ \ell_8 &= 81c^2. \end{aligned}\right.\end{equation} From the coefficient of $x^5y^2z$ in the polynomial $Q_7(t)$, we get $$37 k_4^3+3996 k_8 k_4+7938 k_{12}.$$ So, $k_{12} = 0$.

Consider the case {\bf 6)}. Computing the coefficient of $x^3y^2z$ in the polynomial $Q_7(t)$ gives the vanishing of $k_4^2 + 108k_8$. Hence, we get the case the system of equations (\ref{c_case_equations}) again. From the coefficient of $x^5y^2z$ in the polynomial $Q_7(t)$, we get $$3173 k_4^3+342684 k_8 k_4+71442 k_{12} = 0.$$ So, $k_{12} = 0$. Therefore, this case coincides with {\bf 5)}. 

Cases {\bf 7)} and {\bf 8)} are eliminated for the same reason as case {\bf 2)}.

Consider {\bf 9)}. The coefficient of $x^2yz$ in the polynomial $Q_7(t)$ gives $k_2^2 = 0$. So, this case is eliminated because $k_2\neq 0$.

Thus, for $k_0 = -27$, we obtain the possible cases {\bf 3)} and {\bf 5)} satisfy the necessary conditions for 3-valued algebraic groups. 

Now let us consider the common solution set of the systems (\ref{big_system}) and (\ref{small_system}) when $k_0 = 0$:

\begin{enumerate}[\bf 1)]

\item $k_2=0\land k_6=0\land k_8=0\land k_{10}=0\land \ell_6=0\land \ell_8=\frac{1}{6} \left(k_4 \ell_4-3 k_4^2+2 \ell_4^2\right)$.

\item $k_2=0\land k_4=0\land k_6=0\land k_8=0\land k_{10}=0\land \ell_6=0\land \ell_8=\ell_4^2/3$.

\item $k_6=k_2 k_4\land k_8=0\land k_{10}=0\land \ell_4=k_4\land \ell_6=0\land \ell_8=0$.

\end{enumerate}

Let us consider case {\bf 1)}. Computing the coefficients of $x^4yz$ and $x^2y^3z$ in the polynomial $Q_7(t)$ yields the system
$$\begin{cases}
-5k_{4}\ell_{4}+k_{4}^{2}-10\ell_{4}^{2}+30\ell_{8} & =0,\\
-37k_{4}\ell_{4}+12k_{4}^{2}-74\ell_{4}^{2}+222\ell_{8} & =0,
\end{cases}$$
which has the solution $k_4=0\land \ell_8=\ell_4^2/3$. Computing the coefficient in the $x^5y^2z$ yields $$-27 k_4 \ell_4^2+k_4^2 \ell_4+18 k_4 \ell_8+63 k_{12}-42 \ell_4^3+119 \ell_8 \ell_4 = 0.$$ Substituting $k_4=0\land \ell_8=\ell_4^2/3$ leads to the equation $k_{12} = \ell_4^3/27$. So, we have \begin{equation*}\left\{\begin{aligned}\label{c_case}k_0 &= 0,\\ k_2 &= 0,\\ k_4 &= 0,\\ k_6 &= 0,\\ k_8 &= 0,\\ k_{10} &= 0,\\ k_{12} &= \ell_4^3/27,\\  \ell_6 &= 0,\\ \ell_8 &= \ell_4^3/27. \end{aligned}\right.\end{equation*}    

Case {\bf 2)}. It is contained in the case {\bf 1)}.

Let us consider case {\bf 3)}. Computing the coefficients of $x^2yz$ and $x^3y^2z$ in the polynomial $Q_7(t)$ gives the equalities $k_2 = 0$ and $k_4(10 k_2^2+33 k_4) = 0$. Hence $k_2 = k_4 = 0$, and we arrive at case {\bf 1)}.

Thus, for $k_0 = 0$ we obtain that the only possible case is {\bf 2)}.

Thus, we have the necessary conditions on the coefficients of the universal law. Now, let us prove that they are also sufficient.

For $k_0 = 0$, we get the polynomial $$\left(\sigma_1 + \frac{\ell_4}{3}\sigma_3\right)^3$$ which defines the diagonal law of the 1-valued group from Example \ref{alpha_x_y_z}.

The case $k_0 = -27$ is more interesting. For each $c\in\Cc$, consider a cubic $$\E = \{y^2 = x^3 + c\}$$ in the affine chart $\{z = 1\}\subset\Cc\!P^2$. Recall, for the sum of two points $(x_1, y_1)$ and $(x_2, y_2)$ on $\E$ we have \begin{equation}\label{sum_formula}\left\{\begin{aligned} x_3 &= -x_1 - x_2 + m^2,\nonumber\\ y_3 &= m(x_1 - x_3) - y_1\end{aligned}\right.\end{equation} where $$m = \frac{y_1 - y_2}{x_1 - x_2}.$$ The curve $\E$ has an automorphism $\tau$ of order 3 sending $(x, y)$ to $(\epsilon x, y)$ for $\epsilon = e^{2\pi i/3}$. Indeed, $$\begin{aligned}\tau(x_1, y_2)\oplus \tau(x_2, y_2) &= \left(- \epsilon x_1 - \epsilon x_2 + \left(\frac{y_1 - y_2}{\epsilon x_1 - \epsilon x_2}\right)^2, \frac{y_1 - y_2}{\epsilon x_1-\epsilon x_2}(\epsilon x_1 - \epsilon x_3) - y_1 \right)\\ &= \tau(x_3, y_3).\end{aligned}$$

 There is a 3-fold branched covering $$\begin{aligned}\pi: \E&\to \Cc\!P^1\\ (x, y)&\mapsto y\\\infty&\mapsto\infty\end{aligned}$$ with branch points $\pm\sqrt{c}$, $\infty$. Identify the orbit space $\E/\langle \tau \rangle$ with $\Cc\!P^1$ by the fibers of the projection $\pi$. Hence, write down the resultong 3-valued group $\E_{\langle \tau \rangle}$ on $\Cc\!P^1$ with neutral element $\infty$: $$\begin{aligned}y_1\ast y_2 &= [\pi((x_1, y_1)\oplus(\epsilon^k x_2, y_2))\mid k = 0, 1, 2]\\ &= \left[ m_k\left(2\sqrt[3]{y_1^2 - c} + \epsilon^k\sqrt[3]{y_2^2 - c} - m_k^2\right) - y_1 \right],\end{aligned}$$ where for each $k = 0, 1, 2$ we denote $$m_k = \frac{y_1 - y_2}{\sqrt[3]{y_1^2 - c} - \epsilon^k\sqrt[3]{y_2^2 -c}}.$$

Direct computations by Vieta's formulas show that the elements in the expression for $y_1\ast y_2$ are the roots of the polynomial $(xyz)^3P_{3}(1/z; 1/x, 1/y)$, where $$P_{3}(z; -x, -y) = \sigma_1^3 - 27 \sigma_3 + 18 c \sigma_1^2 \sigma_3 - 54 c \sigma_2 \sigma_3 - 27 c^2 \sigma_2^2 \sigma_3 + 81 c^2 \sigma_1 \sigma_3^2.$$ It is easy to see that any automorphism of $\Cc\!P^1$ maps an algebraic group to some isomorphic algebraic group. Hence, the Möbius transform $x\mapsto 1/x$, $y\mapsto 1/y$, $z\mapsto 1/z$ maps the group $\E_{\langle \tau \rangle}$ to an algebraic 2-valued group with neutral element 0. Restricting to $\Cc$ we get the desired group.

\end{proof}

\section{Kronecker Sums and Frobenius Companion Matrices}\label{Kronecker_sums_and_companions}

Let us recall the following definition, important for what follows (see, for example, \cite[Definition 4.4.4]{Horn}).

\begin{defi}\label{Kronecker_sum}
Let $A$ and $B$ be arbitrary square matrices over the field $\Cc$ of orders $m$ and $n$, respectively. The {\it Kronecker sum} $A\boxplus B$ of the matrices $A$ and $B$ is the $mn\times mn$ matrix
$$A\boxplus B = A\otimes_{\Cc} I_n + I_m\otimes_{\Cc} B.$$
\end{defi}

Note that the concept of Kronecker sum is not as widely represented in standard linear algebra courses as the Kronecker product (tensor product of matrices). Nevertheless, this concept has many applications in various fundamental and applied fields of science, including differential equations, quantum computing, machine learning, robotics, and others. For example, Kronecker sums yield a criterion for the existence and uniqueness of a solution to the matrix equation $AX + XB = C$ \cite[Theorem 4.4.6]{Horn}: the key idea is that this equation can be rewritten as $(B^{T}\boxplus A)\vecc X = \vecc C$, where $\vecc(\cdot)$ denotes the vectorization operation. Such an equation arises in the study of the stability of linear differential systems $\dot{x} = Ax$ with constant coefficients for global asymptotic stability in the sense of Lyapunov \cite[Chapter XV, Section 5, Theorem 3, Theorem 3']{Gantmacher}. 

The following fact \cite{Stephanos} about polynomials composed of tensor products is important for the present work. A more modern exposition can be found in book \cite{Lancaster}.

\begin{prop}[(Section 12.2, Theorem 1 in \cite{Lancaster})]\label{Stephanos}
Let $A\in \Cc^{m\times m}$ and $B\in \Cc^{n\times n}$ be complex matrices, and let $p(x, y) = \sum\limits_{j, k = 0}^{\ell}c_{jk}x^j y^k$ be a complex polynomial. Consider the matrix
$$p(A, B) = \sum\limits_{j, k = 0}^\ell c_{jk}A^j\otimes B^k.$$ 
Let $\lambda_1, \dots, \lambda_m$ be the eigenvalues of the matrix $A$, and $\mu_1, \dots, \mu_n$ the eigenvalues of $B$. Then the eigenvalues of the matrix $p(A, B)$ are the $mn$ numbers $p(\lambda_r, \mu_s)$, where $r = 1, \dots, m$ and $s = 1, \dots, n$.  
\end{prop}

The proof of Proposition \ref{Stephanos} follows from considering the Jordan normal forms of the matrices $A$ and $B$.

For $p(x, y) = x + y$, we obtain:

\begin{prop}[(Section 12.2, Corollary 2 in \cite{Lancaster})]\label{eigen_sum}
Let $A$ and $B$ be arbitrary square matrices over the field $\Cc$ of orders $m$ and $n$, respectively. Let $\lambda_1, \dots, \lambda_m$ and $\mu_1, \dots, \mu_n$ be the eigenvalues with multiplicities for the matrices $A$ and $B$. Then the $mn$-multiset $[\lambda_j + \mu_k]$ of all possible sums is the multiset of eigenvalues of the matrix $A\boxplus B$.  
\end{prop}

We recall the following.

\begin{defi}\label{companion_matrix}
Let $f(t) = t^m + a_{m-1}t^{m-1} + \dots + a_0\in \Cc[t]$. The {\it Frobenius companion matrix} $F(a_{m-1}, ..., a_0)$ of the polynomial $f(t)$ is the matrix
\begin{equation}\label{Frobenius_matrix}
\left(\begin{array}{ccccc}
0 & 0 & 0 & 0 & -a_{0}\\
1 & 0 & 0 & 0 & -a_{1}\\
0 & 1 & 0 & 0 & -a_{2}\\
0 & 0 & \ddots & \vdots & \vdots\\
0 & 0 & 0 & 1 & -a_{m-1}
\end{array}\right).
\end{equation}
Similarly, the {\it block Frobenius companion matrix} is defined for a polynomial $f(t) = t^m + A_{m-1}t^{m-1} + \dots + A_0\in \Mat_{n}(\Cc)[t]$.
\end{defi}

From Proposition \ref{eigen_sum}, we immediately obtain:

\begin{prop}\label{coroll1}
Let $f(t) = t^m + a_{m-1}t^{m-1} + \dots + a_0$ and $g(t) = t^n + b_{n-1}t^{n-1} + \dots + b_0$ be polynomials in the ring $\Cc[t]$ with multisets of roots $[\lambda_j]$ and $[\mu_k]$, respectively. Then the characteristic polynomial
$$\chi(F(a_{n-1}, \dots, a_0)\boxplus F(b_{n-1}, \dots, b_0); t) = \det(t\cdot I - F(a_{n-1}, \dots, a_0)\boxplus F(b_{n-1}, \dots, b_0))$$
in the variable $t$ of the Kronecker sum of their Frobenius companion matrices has as its roots all possible $mn$ sums of the form $\lambda_j + \mu_k$.
\end{prop}

The operation of substituting a polynomial into another polynomial naturally arises in problems related to the Yang–Baxter equation \cite{Krichever81}, multivalued discrete dynamics \cite{Veselov91}, and homomorphisms of $n$-valued groups. We present an expression for the composition $g(f)$ of polynomials $f$ and $g$ in terms of their Frobenius companion matrices $C_f$ and $C_g$.

\begin{prop}[\cite{Carrillo}]\label{C_g_circ_C_f}
Let $f(t) = t^m + a_{m-1}t^{m-1} + \dots + a_0$ and $g(t) = t^n + b_{n-1}t^{n-1} + \dots + b_0$ be polynomials over the ring $\Cc[t]$ with multisets of roots $[\lambda_j]$ and $[\mu_k]$, respectively. Then the following equality holds:
$$g(f)(t) = \det(t\cdot I_{mn} - C_g\circ C_f),$$
where
$$\begin{aligned}
C_g\circ C_f = \left(\begin{array}{ccccc}
0 & 0 & \cdots & 0 & -a_{0}I_{n}+C_{g}\\
I_{n} & 0 & \cdots & 0 & -a_{1}I_{n}\\
0 & I_{n} & \cdots & 0 & -a_{2}I_{n}\\
\vdots & \vdots & \ddots & \vdots & \vdots\\
0 & 0 & \cdots & I_{n} & -a_{m-1}I_{n}
\end{array}\right)
\end{aligned}$$
is the block Frobenius companion matrix of the matrix polynomial
$$q(t) = f(t)I_n - C_g.$$
\end{prop}

We present a direct and short proof. We will need the following lemma.

\begin{lemma}[\cite{Melman}]\label{lemma_blocks}
Let $M_i$, $N_j$ be matrices of size $n\times n$ over the field $\Cc$. Then the determinant of the block matrix
$$\left(\begin{array}{cccccc}
M_{1} & M_{2} & M_{3} & \cdots & \cdots & M_{m}\\
-I & N_{1} & 0 & \cdots & \cdots & 0\\
0 & -I & N_{2} & \cdots & \cdots & 0\\
0 & 0 & -I & \ddots & \cdots & \vdots\\
\vdots & \vdots & \vdots & \ddots & \ddots & 0\\
0 & 0 & 0 & \cdots & -I & N_{m-1}
\end{array}\right)$$
is equal to the determinant of the matrix
$$M_{1}\prod\limits_{j=1}^{m-1}N_{j}+M_{2}\prod\limits_{j=2}^{m-1}N_{j}+\dots+M_{m-1}N_{m-1}+M_{m}.$$
\end{lemma}

\begin{proof}[Proof of Proposition \ref{C_g_circ_C_f}]
Note that since the determinant of a matrix is invariant under reflection with respect to the anti-diagonal, we have the equality
$$h(t) := \det(tI_n - C_g\circ C_f) = \det\left(\begin{array}{ccccc}
t+a_{m-1}I_{n} & a_{2}I_{n} & \cdots & a_{1}I_{n} & a_{0}I_{n}-C_{g}\\
-I_{n} & tI_n & \cdots & 0 & 0\\
0 & -I_{n} & \cdots & tI_n & 0\\
\vdots & \vdots & \ddots & \vdots & \vdots\\
0 & 0 & \cdots & -I_{n} & tI_n
\end{array}\right).$$
By Lemma \ref{lemma_blocks}, we obtain:
\begin{equation}\label{g(f)_composition}
h(t) = \det(f(t)I_n - C_g).
\end{equation}
By a property of Frobenius companion matrices, the monic polynomial $h(t)$ has the same roots as the monic polynomial $g(f(t))$. Therefore, they are equal.
\end{proof}

\section{Groups $\G_n$ and Kronecker Sums}\label{Kronecker}\label{G_n_and_Kronecker}

In this section, we present in Proposition \ref{prop10} a new approach to constructing the $n$-valued group $\G_n$ based on the concept of Kronecker sums of matrices. Theorem \ref{thm1} provides homogeneous symmetric polynomials $p_n(z; x_1, \dots, x_m)$ with integer coefficients that describe the iterations of $n$-valued multiplication in the group $\G_n$.

From the general coset construction in Proposition \ref{coset_theorem}, it follows that for each finite subgroup $G$ of order $n$ of the group of automorphisms of the additive group $(\R^n, +)$ of real numbers, there is an associated $n$-valued group $\R^n(G)$. If no continuity condition is imposed on the automorphisms, the structure of the group $G$ can be quite exotic: for example, $G$ can be defined using Hamel bases. If we consider only continuous automorphisms of $(\R^n, +)$, then $G$ is a finite subgroup of the group $\GL(n, \R)$.

In this work, we are primarily interested in $n$-valued groups on the additive group $(\Cc, +)$ of complex numbers defined by holomorphic functions in some domain $U\subset \Cc$. In \cite{Buchstaber75}, a classification was obtained of all 2-valued groups (not necessarily coset groups) whose multiplication law is given by power series with coefficients in any algebra over a field of characteristic $0$.

Recall that the focus of this work is the family of polynomials $p_n(z; x, y)$ introduced in \cite{Buchstaber_Novikov}:
\begin{equation}\label{p_n}
p_n(z; x,y) = \prod\limits_{w\in \inv(x)\ast \inv(y)} (z - w),
\end{equation}
where $w$ runs over the elements of the $n$-multiset $\inv(x)\ast \inv(y)$ in the sense of the operation $\ast$ on the $n$-valued group $\G_n$ (see Example \ref{G_n_defi}).

\begin{prop}\label{prop10}
For each positive integer number $n$, the following statements hold:

\begin{enumerate}[(i)]

\item The polynomial $p_n(z; x, y)$ from (\ref{p_n}) is symmetric in the variables $x$, $y$, and $z$.

\item For all complex $x$, $y$, and $z$, we have the equality
$$p_n(z^n; x, y) = \chi(F(\underset{n}{\underbrace{0,\dots,0,(-1)^{n+1}x}})\boxplus F(\underset{n}{\underbrace{0,\dots,0,(-1)^{n+1}y}}); z).$$

\item The polynomial $p_n(z; x, y)$ is a homogeneous polynomial in the variables $x$, $y$, and $z$ of degree $n$ with integer coefficients.
\end{enumerate}

\end{prop}

\begin{proof}
Expanding each factor in expression (\ref{p_n}) into $n$ linear terms, we obtain
\begin{equation}\label{eq1}
p_n(z; x, y) = \prod\limits_{r, s = 1}(\sqrt[n]{z} + \epsilon^r\sqrt[n]{x} + \epsilon^s\sqrt[n]{y}),
\end{equation}
where $\varepsilon$ is an (arbitrary) primitive $n$-th root of unity.

\begin{enumerate}[(i)]

\item Recall that we fix the root branch such that $\sqrt[n]{1} = 1$. From the equality (\ref{eq1}), it follows that the expression $p_n(z; x, y)$ is symmetric. Indeed,
\begin{align*}
p_n(x; z, y) &= \prod\limits_{r,s = 1}^n (\sqrt[n]{x} - \epsilon^r\sqrt[n]{(-1)^nz} - \epsilon^s\sqrt[n]{(-1)^ny}) \\
&= \prod\limits_{r,s = 1}^n (\sqrt[n]{x} + \epsilon^r\sqrt[n]{z} + \epsilon^s\sqrt[n]{y}) \\
&= \left(\epsilon \dots \epsilon^{n-1} \right)^n\prod\limits_{r, s =1}^n \epsilon^{-r}(\sqrt[n]{z} + \epsilon^{-r}\sqrt[n]{z} + \epsilon^{s+r}\sqrt[n]{y})\\
&= \prod\limits_{r,s = 1}^n (\sqrt[n]{z} + \epsilon^r\sqrt[n]{x} + \epsilon^s\sqrt[n]{y}) \\
&= p_n(z; x, y),
\end{align*}
where the second equality follows from the fact that $(-\epsilon^s\sqrt[n]{(-1)^nx})^n = x$, i.e., for any $n$, the expression $-\epsilon^s\sqrt[n]{(-1)^nx}$ runs over all roots of the polynomial $t^n - x$ as $s = 1, \dots, n$.

\item Note that from equation (\ref{eq1}), it also follows that $p_n(z; x, y)$ has as its roots all sums of the form $x_j + y_k$, where $\{x_j\}$ and $\{y_k\}$ are the sets of roots of the polynomials $z^n + (-1)^{n+1}x$ and $z^n + (-1)^{n+1}y$. Then from Proposition \ref{coroll1}, it follows that
$$p_n(z^n; x, y) =  \chi(F(\underset{n}{\underbrace{0,\dots,0,(-1)^{n+1}x}})\boxplus F(\underset{n}{\underbrace{0,\dots,0,(-1)^{n+1}y}}); z).$$

\item The statement follows immediately from point (ii) and the fact that the eigenvalues of the matrix $A^n$ are the $n$-th powers of the eigenvalues of the matrix $A$.
\end{enumerate}
\end{proof}

Thus, the result of the $n$-valued multiplication (\ref{G_n_defi}) can be redefined as the multiset of roots of the characteristic polynomial of the matrix
$$F(\underset{n}{\underbrace{0,\dots,0,(-1)^{n+1}x}})\boxplus F(\underset{n}{\underbrace{0,\dots,0,(-1)^{n+1}y}}),$$
which is the Kronecker sum of the Frobenius companion matrices corresponding to the polynomials $t^n - x$ and $t^n - y$.

Proposition \ref{prop10} admits a natural generalization. To this end, we define the polynomials
\begin{equation}\label{gen_p_n}
p_n(z; x_1, \dots, x_m) = \prod\limits_{s_1, \dots, s_m = 1}^n (\sqrt[n]{z} + \epsilon^{s_1}\sqrt[n]{x_1} + \dots + \epsilon^{s_m}\sqrt[n]{x_m}).
\end{equation}

\begin{theorem}\label{thm1}
The following statements hold:
\begin{enumerate}[(i)]

\item The polynomials $p_n(z; x_1, \dots, x_m)$ are symmetric in all variables.

\item For all complex $z$, $x_1, \dots, x_m$, we have the equality
$$p_n(z^n; x_1, \dots, x_m) = \chi(F(\underset{n}{\underbrace{0,\dots,0,(-1)^{n+1}x_1}})\boxplus\dots\boxplus F(\underset{n}{\underbrace{0,\dots,0,(-1)^{n+1}x_m}}); z).$$

\item The polynomials $p_n(z; x_1, \dots, x_m)$ are homogeneous polynomials in the variables $z$, $x_1$, \dots, $x_m$ of degree $n^{m-1}$ with integer coefficients.
\end{enumerate}
\end{theorem}

\begin{proof}
Analogous to the proof of Proposition \ref{prop10}.
\end{proof}

\section{Groups $\G_n$ and  Wendt $(x, y, z)$-matrices}\label{G_n_and_Wendt}
This section establishes a connection between the polynomials $p_n(z; x, y)$ and the determinant of the Wendt matrix, which is known in number theory in the context of Fermat's Last Theorem.

Let us recall the following three definitions (see \cite{Gantmacher, Babenko}).

\begin{defi}
A {\it Toeplitz matrix} $\Toep(a_{n-1}, \dots, a_1, a_0, a_{-1}, \dots, a_{-(n-1)})$ is an $n\times n$ matrix
\begin{equation*}  
\left(\begin{array}{cccccc}
a_{0} & a_{-1} & a_{-2} & \dots & a_{-(n-2)} & a_{-(n-1)}\\
a_{1} & a_{0} & a_{-1} & a_{-2} & \ddots & a_{-(n-2)}\\
a_{2} & a_{1} & a_{0} & a_{-1} & \ddots & \vdots\\
\ddots & a_{2} & a_{1} & a_{0} & \ddots & a_{-2}\\
a_{n-2} & \ddots & \ddots & \ddots & \ddots & a_{-1}\\
a_{n-1} & a_{n-2} & \dots & a_{2} & a_{1} & a_{0}
\end{array}\right).
\end{equation*}
\end{defi}

\begin{defi}
A Toeplitz matrix $\Toep(a_{n-1}, \dots, a_1, a_0, a_{-1}, \dots, a_{-(n-1)})$ is called a {\it circulant} if $a_j = a_{j - n}$ for each $j = 1, \dots, n - 1$. Notation: $\Circ(a_0, \dots, a_{n-1})$.
\end{defi}

\begin{defi}
Consider the circulant matrix $\Circ(a_0, \dots, a_{n-1})$. Multiply all elements above the diagonal by a variable $\theta$. The resulting Toeplitz matrix
$$\Circ_\theta(a_0, \dots, a_{n-1}) = \left(\begin{array}{ccccc}
a_{0} & \theta a_{1} & \theta a_{2} & \cdots & \theta a_{n-1}\\
a_{n-1} & a_{0} & \theta a_{1} & \ddots & \vdots\\
a_{n-2} & a_{n-1} & a_{0} & \ddots & \theta a_{2}\\
\vdots & \ddots & \ddots & \ddots & \theta a_{1}\\
a_{1} & \cdots & a_{n-2} & a_{n-1} & a_{0}
\end{array}\right)$$
and its transpose are called {\it $\theta$-circulants}.
\end{defi}

Toeplitz matrices and $\theta$-circulants naturally arise in the context of the {\it companion ring} of a monic polynomial $f(x)$. This ring is a subring of the matrix ring generated by all matrices $g(C_f)$, where $g$ ranges over all polynomials in $R[x]$ for a commutative ring $R$ with unit, and $C_f$ is a Frobenius companion matrix (see Definition \ref{companion_matrix}). A literature review on this topic can be found in \cite{Mednykh}.

Analyzing the structure of Kronecker sums of Frobenius companion matrices, we obtain the following result:

\begin{theorem}\label{xyz_Wendt_det}
The polynomial $p_n(z; x, y)$ defining the $n$-valued multiplication (\ref{operation2}) is the determinant of a $y$-circulant $n\times n$ matrix:
\begin{equation}\label{Circ_y}
\Circ_{y}\left(w^n + (-1)^{n+1}x + y, \binom{n}{1}w, \dots, \binom{n}{n-1}w^{n-1}\right),
\end{equation}
where $w^n = z$.
\end{theorem}

\begin{proof}
We use the item (ii) of Proposition \ref{prop10}. Note that the matrix
$$C = wI - F(\underset{n}{\underbrace{0,\dots,0,(-1)^{n+1}x}})\boxplus F(\underset{n}{\underbrace{0,\dots,0,(-1)^{n+1}y}})$$
is a block $n\times n$ matrix of the form
\begin{equation}\label{block_matrix}
\left(\begin{array}{ccccc}
A & 0 & 0 & \cdots & (-1)^{n+1}xI\\
-I & A & 0 & \cdots & 0\\
0 & -I & A & 0 & 0\\
0 & \ddots & \ddots & \ddots & 0\\
0 & 0 & 0 & -I & A
\end{array}\right),
\end{equation}
where
\begin{equation}\label{matrix_A}
A = wI - F(\underset{n}{\underbrace{0,\dots,0,(-1)^{n+1}y}}).
\end{equation}

Then by Lemma \ref{lemma_blocks}, the determinant of matrix (\ref{block_matrix}) is equal to $\det(A^n + (-1)^{n+1}xI)$. Expanding the Newton binomial $A^n = (wI - F)^n$, we obtain a Toeplitz matrix differing from matrix (\ref{Circ_y}) only by a sign alternation on every odd diagonal (with the main diagonal being numbered 0). The determinant of such a matrix equals the determinant of matrix (\ref{Circ_y}). Indeed, this follows from a general simple fact that in each product $(-1)^{\sgn(\sigma)}a_{1\sigma(1)}a_{2\sigma(2)}\dots a_{n\sigma(n)}$ from the determinant expansion $\det(a_{ij})$ of a matrix $(a_{ij})$, the number of factors $a_{j\sigma(j)}$ such that $j + \sigma(j)$ is odd is even, since the total sum $\sum\limits_{j=1}^n (j+\sigma(j)) = 2n$ is even.     
\end{proof}

\begin{example}
Using Theorem \ref{xyz_Wendt_det}, let us write expressions for $p_n = p_n(z; x, y)$ via determinants of $n\times n$ matrices for $n = 2, 3$, and 4. In each case, $w^n = z$.
$$p_2 = \left|\begin{array}{cc}
w^{2}-x+y & 2wy\\
2w & w^{2}-x+y
\end{array}\right|,
\quad
p_3 = \left|\begin{array}{ccc}
w^{3}+x+y & 3yw & 3yw^{2}\\
3w^{2} & w^{3}+x+y & 3yw\\
3w & 3w^{2} & w^{3}+x+y
\end{array}\right|,$$
$$
p_4 = \left|\begin{array}{cccc}
w^{4}-x+y & 4yw & 6yw^{2} & 4yw^{3}\\
4w^{3} & w^{4}-x+y & 4yw & 6yw^{2}\\
6w^{2} & 4w^{3} & w^{4}-x+y & 4yw\\
4w & 6w^{2} & 4w^{3} & w^{4}-x+y
\end{array}\right|.
$$
\end{example}

\begin{prop}\label{div_n^4}
For prime $n \geqslant 5$, the polynomial
$$p_n(z; x, y) - (x + y + z)^n$$
is divisible by $n^4xyz$.
\end{prop}

\begin{proof}
From the symmetry of the considered polynomials, divisibility by $xyz$ is immediate.

To prove divisibility by $n^4$, we use Theorem \ref{xyz_Wendt_det}. By the definition of the determinant, in the sum
\begin{equation}\label{det_sum}
\sum\limits_{\pi\in S_n} \sgn(\pi) a_{1\pi(1)}\dots a_{n\pi(n)},
\end{equation}
where $(a_{ij})$ is matrix (\ref{Circ_y}), each term containing at least four off-diagonal entries is divisible by $n^4$, since $\binom{n}{k}$ is divisible by $n$ for each $k = 1, \dots, n-1$ and prime $n$.

Now consider the sum of all terms in (\ref{det_sum}), where each term contains exactly three off-diagonal elements of the matrix $(a_{ij})$. Each such term appears with a plus sign and is divisible by $n^3$. There are $2\cdot\binom{n}{3}$ such terms. Therefore, this sum is divisible by $n\cdot \binom{n}{3}$, and hence by $n^4$.

It remains to consider the sum of all terms in (\ref{det_sum}) where each term contains exactly two off-diagonal elements of the matrix $(a_{ij})$. Each such term appears with a minus sign. This sum equals
$$-(z + (-1)^{n+1}x + y)^{n-2}yz\cdot \sum\limits_{k=1}^{n-1} k \binom{n}{k}^2.$$
Some elementary algebraic or combinatorial arguments show (see, for example, \cite{MS_binomial}) that
\begin{equation}\label{Wolstenholme}
\sum\limits_{k=1}^{n-1} k \binom{n}{k}^2 = n\left ( \binom{2n-1}{n-1} - 1 \right ).
\end{equation}
By Wolstenholme's theorem \cite{Wolstenholme62}, the difference
$$\binom{2n-1}{n-1} - 1$$
is divisible by $n^3$ for primes $n \geqslant 5$, so (\ref{Wolstenholme}) is divisible by $n^4$.      
\end{proof}

The result of Proposition \ref{div_n^4} cannot be strengthened, since
$$p_5(z; x, y) = \sigma_1^5 - 5^4\sigma_1^2\sigma_3 + 5^5\sigma_2\sigma_3,$$
where, as usual, $\sigma_1$, $\sigma_2$, $\sigma_3$ are elementary symmetric polynomials in the variables $x$, $y$, and $z$.

We see that in the proof of Theorem \ref{div_n^4}, the divisibility of the difference $\binom{2n-1}{n-1} - 1$ by $n^3$ was naturally required. In this context, let us recall that a Wolstenholme prime is a prime $n$ such that $\binom{2n-1}{n-1} - 1$ is divisible by $n^4$ \cite{MathWorld_Wolstenholme}. Currently, only two Wolstenholme primes are known: 16843 and 2124679. The next Wolstenholme prime, if it exists, must exceed $10^{11}$.

\begin{example}
As noted in the Introduction, computer experiments show that arbitrarily large prime numbers appear in the expansions of the coefficients of the polynomials $p_n(z; x, y)$ in the basis of elementary symmetric polynomials. To illustrate, we present the explicit form of the polynomial $p_{18}(z; x, y)$ in the basis of elementary symmetric polynomials $\sigma_1, \sigma_2$, and $\sigma_3$. For example, the notation $(5,0,1)\to -2^8 3^1 67^1$ means that the coefficient of $\sigma_1^5\sigma_3$ is $-2^8\cdot 3\cdot 67$.

\begin{flalign*}
&(18,0,0)\to  1\\
&(16,1,0)\to - 2^2 3^2 \\
&(15,0,1)\to - 2^1 3^5 43^1 293^1 13339^1 \\
&(14,2,0)\to 2^6 3^2 \\
&(13,1,1)\to - 2^1 3^5 39079^1 30478663^1 \\
\end{flalign*}
\begin{flalign*}
&(12,3,0)\to - 2^8 3^1 7^1 \\
&(12,0,2)\to 3^4 167^1 58369^1 1702940402507^1 \\
&(11,2,1)\to - 2^2 3^6 57769225879741^1 \\
&(10,4,0)\to 2^9 3^2 7^1 \\
&(10,1,2)\to - 2^2 3^5 7^1 97^1 36913^1 180317^1 3375001057^1 \\
&(9,3,1)\to - 2^1 3^5 23^1 144589^1 5245247209^1 \\
&(9,0,3)\to - 2^1 3^8 5^1 12713^1 76919^1 37764598382689403^1 \\
&(8,5,0)\to - 2^{11} 3^2 7^1 \\
&(8,2,2)\to 3^5 67589^1 626540941303495210351^1 \\
&(7,4,1)\to - 2^5 3^5 11^1 13^1 102997724923217^1 \\
&(7,1,3)\to - 2^1 3^8 13^1 124035886813^1 453195935961757643^1 \\
&(6,6,0)\to 2^{14} 3^1 7^1 \\
&(6,3,2)\to - 2^1 3^4 13^3 1087^1 1879^1 2833^1 528719679255133^1 \\
&(6,0,4)\to 2^1 3^7 5^1 13^1 113^1 86137^1 215724736933^1 207036951417917^1 \\
&(5,5,1)\to - 2^6 3^6 10168829241424199^1 \\
&(5,2,3)\to - 2^1 3^9 5^1 367^1 739^1 466897^1 110336732972567120113^1 \\
&(4,7,0)\to - 2^{16} 3^2 \\
&(4,4,2)\to 2^1 3^5 97^1 683^1 875241448225705706391329^1 \\
&(4,1,4)\to - 2^1 3^8 11^1 179^1 13499^1 19801^1 67601^1 99257^1 1129433^1 1123012127^1 \\
&(3,6,1)\to - 2^9 3^5 11^1 414927770423911^1 \\
&(3,3,3)\to - 2^2 3^8 107137^1 30887295467839157373255894019^1 \\
&(3,0,5)\to - 2^2 3^{11} 17443^1 3104015062391^1 839030750625213207689^1 \\
&(2,8,0)\to 2^{16} 3^2 \\
&(2,5,2)\to - 2^2 3^5 41^1 5691615916625701258286918167^1 \\
&(2,2,4)\to 3^8 7^2 137^1 141017479^1 2779127063107^1 131095595871761^1 \\
&(1,7,1)\to - 2^9 3^5 721117^1 1512997111^1 \\
&(1,4,3)\to - 2^1 3^8 5^1 199^1 520747^1 2094293950849^1 19804297603859^1 \\
&(1,1,5)\to - 2^1 3^{11} 5^1 23^1 109^1 163^1 271^1 2269^1 5779^1 58049^1 2951599681331246837^1 \\
&(0,9,0)\to - 2^{18} \\
&(0,6,2)\to 3^4 10837^1 8379438461^1 73146705440157233^1 \\
&(0,3,4)\to - 2^1 3^7 11^1 443^1 105199^1 9893951^1 115291956551^1 4149469127033^1 \\
&(0,0,6)\to 3^9 109^3 163^3 271^3 2269^3 5779^3 \\
\end{flalign*}
\end{example}

The proof of Theorem~\ref{xyz_Wendt_det} can be iterated to express the polynomial $p_n(z; x_1, \dots, x_m)$, with $m > 2$, as the determinant of a matrix of order $n^{m-1}$.

\begin{theorem}\label{p_n_via_A^n}
The following statements hold:

\begin{enumerate}[(i)]

\item There is the equality $$p_n(z; x_1, \dots, x_m) = \det(A^n + (-1)^{n+1}x_1I),$$ where $$A = wI - F(\underset{n}{\underbrace{0,\dots,0, x_2}})\boxplus \dots\boxplus F(\underset{n}{\underbrace{0,\dots,0, x_m}})$$ and $w^n = z$. 

\item All the minus signs in the off-diagonal elements of the matrix $B = A^n + (-1)^{n+1}x_1$ can be replaced with plus signs—this does not change the determinant.

\item The matrix $B$ is symmetric with respect to the anti-diagonal.

\item For $x_2 = \dots = x_m = z = 1$ and $x_1 = (-1)^n$, the matrix $W_{m, n}$ is obtained, whose elements are zeros and multinomial coefficients in the polynomial $(z + x_2 + \dots + x_m)^n$. The matrix $W_{m,n}$ is orthosymmetric, i.e., symmetric with respect to both the main diagonal and anti-diagonal.

\end{enumerate}
 
\end{theorem}

\begin{example}
Let us write the polynomial $p_2(z; x, y, u)$ for $n = 2$ (as usual, $w^n = z$):

$$
\left|
\begin{array}{cccc}
 u+w^2-x+y & 2 u w & 2 w y & 2 u y \\
 2 w & u+w^2-x+y & 2 y & 2 w y \\
 2 w & 2 u & u+w^2-x+y & 2 u w \\
 2 & 2 w & 2 w & u+w^2-x+y \\
\end{array}
\right|.
$$

\end{example}

\begin{example}
Let us write several polynomials $p_n = p_n(z; x_1, \dots, x_m)$ for $m = 3$ via elementary symmetric functions $\sigma_k$, $k = 1, \dots, 4$. 
\begin{align*} 
p_1 &= \sigma_1, \\ 
p_2 &= \sigma_1^4 - 2^3 \sigma_1^2 \sigma_2 + 2^4 \sigma_2^2 - 2^6 \sigma_4, \\ 
p_3 &= \sigma_1^9 - 3^4 \sigma_1^6 \sigma_3 + 3^7 \sigma_1^3 \sigma_3^2 - 3^9\sigma_3^3 + 2\cdot 3^7 \sigma_1^5 \sigma_4 - 3^9 \sigma_1^3 \sigma_2 \sigma_4 + 3^{10} \sigma_1^2 \sigma_3 \sigma_4.
\end{align*}

\end{example}

\begin{example}
The same for $m = 4$:
\begin{align*} 
p_1 &= \sigma_1, \\ 
p_2 &= \sigma_1^8 - 2^4 \sigma_1^6 \sigma_2 + 2^5 \cdot 3 \sigma_1^4 \sigma_2^2 
- 2^8 \sigma_1^2 \sigma_2^3 + 2^8 \sigma_2^4 
- 2^7 \sigma_1^4 \sigma_4 + 2^{10} \sigma_1^2 \sigma_2 \sigma_4 
- 2^{11} \sigma_2^2 \sigma_4 \\ &+ 2^{12} \sigma_4^2 
- 2^{11} \sigma_1^3 \sigma_5 + 2^{13} \sigma_1 \sigma_2 \sigma_5 
- 2^{14} \sigma_3 \sigma_5. 
\end{align*} 
\end{example}

From Theorem \ref{xyz_Wendt_det}, the following connection with an important class of determinants in number theory also follows.

\begin{defi}
The {\it Wendt matrix $W_n$} is a circulant matrix
\begin{equation}\label{W_n_matrix} 
\left(\begin{array}{ccccc}
1 & \binom{n}{1} & \binom{n}{2} & \cdots & \binom{n}{n-1}\\
\binom{n}{n-1} & 1 & \binom{n}{1} & \cdots & \binom{n}{n-2}\\
\vdots & \vdots & \vdots & \ddots & \vdots\\
\binom{n}{1} & \binom{n}{2} & \binom{n}{3} & \cdots & 1
\end{array}\right). 
\end{equation}
\end{defi}

Note that matrix (\ref{W_n_matrix}) is symmetric with respect to both the main diagonal and anti-diagonal. Such matrices are called {\it orthosymmetric} in \cite{Wendt}.

Thus, we have obtained an $(x, y, z)$-parametric family of matrices, which in the case $x = (-1)^n$, $y = z = 1$ coincides with the matrix $W_n$. 

All of the above motivates the following:

\begin{defi}
 Matrix (\ref{Circ_y}) is called the {\it Wendt $(x, y, z)$-matrix}.
\end{defi}

Wendt's (1984) approach to Fermat's Last Theorem based on his determinant is related to solving congruences of the form 
$$x^p + y^p + z^p \equiv 0\  (\mathrm{mod} \ q)$$
for different odd primes $p$ and $q$ \cite{Wendt}. A modern exposition of the results of this work can be found in \cite{Fermat_for_amateurs}.

\begin{prop}[\cite{Wendt}]\label{Wendt_criterion}
Let $p$ be an odd prime and $q = 2kp+1$ be prime for some $k \geqslant 1$. Then there exist integers $x, y, z$, not divisible by $q$, such that $x^p + y^p + z^p \equiv 0$ (mod $q$) if and only if $q$ divides $\det W_{2k}$.
\end{prop}

The main idea of the proof of Proposition \ref{Wendt_criterion} is that the determinant of the Wendt matrix $W_n$ is expressed \cite[Chapter IV, {\bf (2a)}]{Fermat_for_amateurs} via the resultant $\res_t(\cdot, \cdot)$ of the polynomials $(1+t)^n - t^n$ and $t^n - 1$ in the variable $t$. This is a particular manifestation of the well-known fact \cite[Chapter IV, Lemma 2.1]{Fermat_for_amateurs}:
$$\det \Circ(a_0, \dots, a_{n-1}) = \res_{t}(a_0 + \dots + a_{n-1}t^{n-1}, t^n - 1).$$

For the polynomials $p_n(z; x, y)$, we get an analogous result:

\begin{prop}\label{res_p_n}
The following equality holds
$$p_n(z; x, y) = (-1)^{n} \res_{t}(Q(t), t^n - 1),$$
where $Q(t) = z - (u + vt)^n$, $u^n = (-1)^nx$ and $v^n = (-1)^ny$.
\end{prop}

\begin{proof}
Recall that the resultant of two polynomials is the product of the values of one polynomial at the roots of the other. We obtain the required statement from formula (\ref{p_n}). 
\end{proof}

For $x = (-1)^n$,  $y = z = 1$, from Proposition \ref{res_p_n}, we get:
$$\det W_n = (-1)^n\res_t(1 + (1+t)^n, t^n - 1) = \res((1+t)^n - t^n, t^n - 1).$$ 

Note that the coefficients of the polynomial $Q(t)$ are not polynomials in $x$ and $y$. Nevertheless, in Theorem \ref{discriminant_two_dim}, we will show that each polynomial $p_n(z; x, y)$ is the discriminant of some polynomial in $t$, whose coefficients are integer polynomials in $x$ and $y$. 

The main result of this section is a new formulation of Theorem \ref{xyz_Wendt_det}:

\begin{theoremprime}[{\bf 4'}]\label{xyz_Wendt_and_p_n}
For each integer $n\geqslant 1$, the polynomial $p_n(z; x, y)$ is the determinant of the  Wendt $(x, y, z)$-matrix (\ref{Circ_y}).
\end{theoremprime}

\section{Irreducibility of Polynomials $p_n$}\label{nonderogatory}

The main result of this section is the proof of the irreducibility of the polynomials $p_n$ over different fields.

\begin{theorem}\label{trans_extension}
Let $x_1\dots x_m$ be algebraically independent variables and $n_1\dots n_m$ be positive integer numbers. Then the polynomial
\begin{equation}\label{sum_of_Frobenius_cells}q(\{n_i\}, \{x_j\}; z) =  \chi(F(\underset{n_1}{\underbrace{0,…,0, x_1}})\boxplus…\boxplus F(\underset{n_m}{\underbrace{0,…,0, x_m}}); z)\end{equation}
in the variables $z, x_1\dots x_m$ is irreducible over any field $\kk$ with $\text{gcd}(n, e(\kk)) = 1$, where $$F(\underset{n_j}{\underbrace{0,…,0, x_j}})$$ denotes Frobenius companion matrix (\ref{Frobenius_matrix}) for the polynomial $t^{n_j} = x_j$, $n = \text{lcm}(n_1\dots n_m)$, and $e(\kk) = 1$ if $\charr\kk = 0$ and $e(\kk) = \charr\kk$ otherwise.
\end{theorem}

\begin{proof}
Let $L$ be the field of rational functions $\kk(x_1\dots x_m)$ in algebraically independent variables $x_1\dots x_m$, i.e., the field $L$ is obtained from the field $\kk$ by $m$ successive transcendental extensions. Consider the algebraic extension of fields
\begin{equation}\label{degree_irred_p_n}L\subset L(x_1^{1/n_1}\dots x_m^{1/n_m}).\end{equation}

Consider the extension $L\subset L(x_1^{1/n_1})\cong \kk(x_1^{1/n_1}, x_2\dots x_m)$ and check that the conditions of the classical result \cite[Chapter VI, Theorem 9.1]{Lang} are satisfied, which implies that this extension has degree $n_1$, i.e., the polynomial $t^{n_1} - x_1$ is irreducible over $L$. The conditions \cite[Chapter VI, Theorem 9.1]{Lang} require that $x_1 \notin L^p$ and $x_1 \notin -4L^4$ when $4 \mid n_1$. Assume that $x_1 = C\cdot (R/S)^p$ for some positive integer $p$, polynomials $R, S \in \kk[x_1\dots x_m]$, and a nonzero element $C \in \kk$. Without loss of generality, assume that $R$ and $S$ do not have a common factor of $x_1$. Then the equation $CR^p = x_1S^p$ leads to a contradiction.

Since all the variables $x_j$ are algebraically independent, the analysis of the next level of the tower proceeds similarly: replacing the variable $x_1 := x_1^{1/n_1}$ reduces the problem of the irreducibility of the polynomial $t^{n_2} - x_2$ over the field $L(x_1^{1/n_1}, x_2\dots x_m)$ to the previous reasoning. Repeating this step at each level, we find that the degree of the extension (\ref{degree_irred_p_n}) is $n_1 \dots n_m$ by the tower rule for field extensions.

It is clear that no polynomials $R$ and $S$ from $L$ satisfy $S^{n_j}x_i = R^{n_i}x_j$ for $i \neq j$. This means that the conditions \cite[Corollary 4.7]{Albu} are satisfied in our case, so the element $\theta = x_1^{1/n_1} + … + x_m^{1/n_m}$ is primitive: $$L(x_1^{1/n_1}\dots x_m^{1/n_m}) = L(\theta).$$

The degree of the polynomial $q(\{n_i\}, \{x_j\}; z)$ is $n_1 \dots n_m$, its coefficients lie in the ring $\kk[x_1\dots x_m]$, since it is the characteristic polynomial of the matrix (\ref{sum_of_Frobenius_cells}), and it annihilates the element $\theta$. Thus, the polynomial $q(\{n_i\}, \{x_j\}; z)$ is irreducible over the field $L$, and therefore also over the field $\kk \subset L$.
\end{proof}

For $n_1 = … = n_m = n$ and $\kk = \Cc$, we get

\begin{coroll}\label{irred_p_n_over_C}
The polynomials $p_n(z; x_1\dots x_m)$ in the variables $z_1\dots x_1\dots x_m$ are irreducible over $\Cc$ for each positive integer $n$.
\end{coroll}

\begin{theorem}\label{irreducible}
Let $m$, $n_1$\dots $n_m$, $k_1$\dots $k_m$ be positive integers. Let the positive integers $x_1$\dots $x_m$ be pairwise coprime, and also such that $n_j$ is the smallest positive integer for which $a_j := \left(x_j^{1/k_j}\right)^{n_j}$ is a positive integer for all $j = 1\dots m$. Then the polynomial
\begin{equation}\label{sum_of_Frobenius_cells_a}q(\{n_i\}, \{a_j\}; z) =  \chi(F(\underset{n_1}{\underbrace{0,…,0, a_1}})\boxplus…\boxplus F(\underset{n_m}{\underbrace{0,…,0, a_m}}); z)\end{equation}
in the variable $z$ is irreducible over $\Q$.
\end{theorem}

\begin{proof}
By \cite[Theorem]{Ursell}, the degree of the extension $$\Q\subset \Q[x_1^{1/k_1}\dots x_m^{1/k_m}]$$ is $n_1 \dots n_m$. By \cite[Corollary 4.7]{Albu}, this extension is generated by the element $\theta = x_1^{1/k_1} + … + x_m^{1/k_m}$. Thus, the degree of the element $\theta$ is $n_1 \dots n_m$. On the other hand, by Proposition \ref{coroll1}, the element $\theta$ is a root of the polynomial $q(\{n_i\}, \{a_j\}; z)$ of degree $n_1 \dots n_m$. Thus, the minimal polynomial of the element $\theta$ is precisely $q(\{n_i\}, \{a_j\}; z)$.
\end{proof}

\begin{coroll}
Fix a positive integer $n \geqslant 2$. Let the positive integer pairwise coprime numbers $x_1\dots x_m$ be such that $(\sqrt[n]{x_j})^k$ is irrational for all positive integer $k < n$ for each $j = 1\dots m$. Then the polynomial $p(z; x_1\dots x_m)$ in the variable $z$ is irreducible over $\Q$.
\end{coroll}

\begin{proof}
The conditions of this corollary satisfy the conditions of Theorem \ref{irreducible}. Therefore, the polynomial $p_n(z^n; x_1\dots x_m)$ in the variable $z$ is irreducible over $\Q$. It immediately follows that the polynomial $p_n(z; x_1\dots x_m)$ is irreducible over $\Q$.
\end{proof}

\section{Polynomials $p_n$ and Discriminant Varieties}\label{p_n_and_disc_varieties}

Let us recall the concept of the resultant of a system of algebraic homogeneous equations, following \cite[Chapter 13]{Gelfand} and \cite[Chapter 3]{Cox}. Consider the set of homogeneous polynomials $F_1\dots F_n$ of degree $d$ in variables $x_1\dots x_n$ with variable coefficients $\{u_{i, \alpha}\}$:
$$F_i = \sum\limits_{\deg \alpha = d} u_{i, \alpha}x^{\alpha}.$$

\begin{prop}[(Chapter 3, Theorem 2.3 \cite{Cox})]
Let $d_1\dots d_n$ be fixed positive numbers. Then there exists a unique homogeneous polynomial $\res_{d_1\dots d_n} = \res \in \mathbb{Z}[u_{i, \alpha}]$ in the coefficient space, satisfying the following properties:

\begin{enumerate}[(i)]
\item If $F_1\dots F_n \in \mathbb{C}[x_1\dots x_n]$ are homogeneous polynomials of degrees $d_1\dots d_n$, then the system $F_1 = … = F_n = 0$ has a nontrivial solution $\Leftrightarrow \res(F_1\dots F_n) = 0$.
\item $\res(x_1^{d_1}\dots x_n^{d_n}) = 1$.
\item The polynomial $\res$ is irreducible even in $\mathbb{C}[u_{i, \alpha}]$.
\end{enumerate}
\end{prop}

\begin{defi}
The given polynomial $\res$ is called the {\it resultant of the space of homogeneous polynomials} of degrees $d_1\dots d_n$. The {\it resultant $\res(F_1\dots F_n)$ of specific polynomials} of degrees $d_1\dots d_n$ is defined as the resultant obtained by substituting the coefficients of the polynomials $F_1\dots F_n$ into the resultant $\res$.
\end{defi}

The homogeneous degree of the polynomial $\res$ is given by
\begin{equation}\label{res}
\deg(\res) = \sum\limits_{j=1}^n d_1 \dots \hat{d_j} \dots d_n,
\end{equation}
where $\hat{d_j}$ denotes the omission of the variable $d_j$.

By analogy with the case of one variable, we can introduce the concept of the discriminant of a homogeneous polynomial in several variables.

\begin{defi}
The {\it discriminant} $\disc_{d, n} = \disc$ of the space of homogeneous polynomials of fixed degree $d$ in variables $x_1\dots x_n$ is the polynomial
$$\disc = \res\left(\frac{\partial{F}}{\partial{x_1}}\dots \frac{\partial{F}}{\partial{x_n}}\right)$$
in the coefficient space $\mathbb{C}[u_{i, \alpha}]$.
\end{defi}

From formula (\ref{res}), it follows that the homogeneous degree of the discriminant of the space of homogeneous polynomials of degree $d$ in variables $x_1\dots x_n$ is equal to $n(d-1)^{n-1}$.

\begin{example}
The discriminant of the space of homogeneous polynomials of degree $d$ in two variables coincides with the discriminant of a general polynomial of degree $d$ in one variable.
\end{example}

The polynomials $p_n(z; x, y)$, which define multiplication in the $n$-valued group $\G_n$, naturally arose in a recent paper \cite{Gaiur} in connection with Bessel's kernel $$K_n = \sum\limits_{i, k} \binom{i + k}{k} \frac{x^i y^k}{z^{i+k}}$$ of solutions of the Picard–Fuchs differential equations.

It was shown \cite[Theorem 2.3]{Gaiur} that the singular variety of solutions to the corresponding differential equations is the union of the set of zeros of the polynomials $p_n$ and the coordinate planes. It was noted that this singular variety is the set of zeros of the discriminant of some polynomial, the form of which we present, correcting sign inaccuracies in the original formulation:

\begin{theorem}\label{discriminant_two_dim}
The discriminant $\disc(P_{x, y, z})$ of the polynomial
$$P_{x, y, z}(T) = ((-1)^n x T^{n-1} + (-1)^n y)(1 + T)^{n-1} - T^{n-1} z$$
with respect to the variable $T$, which is a polynomial of degree $4n - 6$, is related to the polynomial $p_n(z; x, y)$ as follows:
$$(-1)^n (n-1)^{2n-2} (xyz)^{n-2} p_n(z; x, y) = \disc_T(P_{x, y, z}(T))$$
for each $n \geqslant 2$. Equivalently, the polynomial $\disc(P_{x, y, z})$ is defined as the discriminant in the sense of \cite{Gelfand} of the homogenized polynomial
$$(u+v)^{n-1} (x u^{n-1} + y v^{n-1}) - (u v)^{n-1} z$$
with respect to the variables $u$ and $v$.
\end{theorem}

\begin{proof}
As known \cite[Theorem 1.3.3]{Prasolov_polynomials}, for any complex polynomial $f(x) = \sum\limits_{k=0}^n a_k x^k$, its discriminant is expressed by the formula
$$\disc(f) = \frac{1}{a_n} \res(f, f’) = n^n a_n^{n-1} \prod\limits_{\tau: f’(\tau) = 0} f(\tau).$$

We first compute the derivative of the polynomial $P(T) = P_{x, y, z}(T)$:
\begin{align}\label{derivative}
P’(T) = (n-1)\cdot\Big(&(-1)^n x T^{n-2} (1 + T)^{n-1} + (-1)^n x T^{n-1} (1 + T)^{n-2} \\
&\notag + (-1)^n y (1 + T)^{n-2} - z T^{n-2} \Big).
\end{align}
If $\tau$ is a root of the equation $P’(\tau) = 0$, then by expressing the sum of the first and last terms from (\ref{derivative}) multiplied by $\tau$, we get
\begin{equation}\label{deriv_expr}
(-1)^n x \tau^{n-1} (1 + \tau)^{n-1} - \tau^{n-1} z = -(-1)^n x \tau^n (1 + \tau)^{n-2} - (-1)^n y \tau (1 + \tau)^{n-2}.
\end{equation}

Let
\begin{equation}\label{const}
C = (-1)^n (2n - 2)^{2n - 2} x^{2n - 3}.
\end{equation}
Then, taking into account (\ref{deriv_expr}), we have for the polynomial $P(T) = P_{x, y, z}(T)$:
\begin{align}\label{disc_eq}
\frac{1}{C} \disc_T(P_{x, y, z}(T)) &= \prod\limits_{\tau: P’(\tau) = 0} (-1)^n \left( -x \tau^n (1 + \tau)^{n-2} - y \tau (1 + \tau)^{n-2} + y (1 + \tau)^{n-1} \right) \\
&= \prod\limits_{\tau: P’(\tau) = 0} (-1)^n x (1 + \tau)^{n-2} \left(\frac{y}{x} - \tau^n\right) \notag \\
&= (-1)^{n(2n-3)} \left( \prod\limits_{\tau: P’(\tau) = 0} x \right) \left( \prod\limits_{\tau: P’(\tau) = 0} (1 + \tau)^{n-2} \right) \left( \prod\limits_{\tau: P’(\tau) = 0} \left( \frac{y}{x} - \tau^n \right) \right). \notag
\end{align}

We observe that everything reduces to computing three products in (\ref{disc_eq}).

Since $\deg(P’) = 2n - 3$, the first product is
\begin{equation}\label{prod_1}
\prod\limits_{\tau: P’(\tau) = 0} x = x^{2n-3}.
\end{equation}

Compute the second product:
\begin{align}\label{prod_2}
\prod\limits_{\tau: P’(\tau) = 0} (1 + \tau)^{n-2} &= \left((-1)^{2n-3} \prod\limits_{\tau: P’(\tau) = 0}  (-x - \tau)|_{x = 1}\right)^{n-2} \\
&= (-1)^{n(2n-3)} \left( \frac{P’(-1)}{(n-1)\cdot 2x} \right)^{n-2} \notag \\
&= (-1)^{n(2n-3)} \left( \frac{z}{2x} \right)^{n-2}. \notag
\end{align}

Now, compute the third product:
 \begin{align}\label{prod_3}
\prod_{\tau: P'(\tau) = 0} \left(\frac{y}{x} - \tau^n\right)
&= \prod_{\substack{\tau: P'(\tau) = 0 \\ \epsilon : \epsilon^n = 1}} \left(\epsilon\sqrt[n]{\frac{y}{x}} - \tau\right) \displaybreak[1]\\
&= \prod_{\substack{\tau: P'(\tau) = 0 \\ \epsilon : \epsilon^n = 1}} \left.\left(T - \tau\right)\right|_{T = \epsilon\sqrt[n]{\frac{y}{x}}} \displaybreak[1]\notag\\
&= \prod_{\epsilon: \epsilon^n = 1} \frac{P'(\epsilon\sqrt[n]{\frac{y}{x}})}{(n-1)\cdot (-1)^n 2x} \displaybreak[1]\notag\\
&
   \begin{aligned}
    =\prod_{\epsilon: \epsilon^n = 1} \frac{T^n}{2x}&\Big(xT^{-2}(1+T)^{n-1} + xT^{-1}(1+T)^{n-2}\notag \\
   &+ x(1+T)^{n-2} - (-1)^n z T^{-2}\Big)\Big|_{T = \sqrt[n]{\frac{y}{x}}}
   \end{aligned}
    \displaybreak[1]\notag\\
&= \prod_{\epsilon: \epsilon^n = 1} \frac{y}{2x^2}
    \left.\left( \frac{x(1+T)^n}{T^2} - (-1)^n\frac{z}{T^2} \right)\right|_{T = \sqrt[n]{\frac{y}{x}}} \displaybreak[1]\notag\\
&= \frac{y^{n-2}}{2^n x^{2n-2}} \prod_{\epsilon: \epsilon^n = 1} 
   \left( z -((-1)^n\sqrt[n]{x} + (-1)^n\epsilon\sqrt[n]{y})^n \right).\notag
\end{align}

Multiplying the values of (\ref{prod_1}), (\ref{prod_2}), (\ref{prod_3}) of the three obtained products, and not forgetting the factor (\ref{const}), we find from (\ref{disc_eq}):
\begin{equation*}
\begin{aligned}
&(-1)^n(2n-2)^{2n-2}x^{2n-3}\cdot x^{2n-3}\cdot (-1)^{n(2n-3)}\cdot (-1)^{n(2n-3)}\cdot \left(\frac{z}{2x}\right)^{n-2}\cdot\frac{y^{n-2}}{2^nx^{2n-2}} \\
&\cdot\prod\limits_{\epsilon: \epsilon^n = 1}\left( z - ((-1)^n\sqrt[n]{x} + (-1)^n\epsilon\sqrt[n]{y})^n  \right ) \\
&= (-1)^n(n-1)^{2n-2}(xyz)^{n-2}p_n(z).
\end{aligned}
\end{equation*}

\end{proof}

From this result, we obtain

\begin{coroll}[(Proposition 1, \cite{Helou})]
For each odd $n\geqslant 3$, the discriminant $\disc Q(t)$ of the polynomial $Q(t) = (1 + t)^n - t^n - 1$ and the determinant $\det W_{n-1}$ of the Wendt matrix are related by the following equality:
$$\disc Q(t) = (-1)^{\frac{n-1}{2}}n^{n-2}\det W_{n-1}.$$
\end{coroll}

Consider the homogeneous polynomial
\begin{equation}\label{homogeneous_mult}
P(u) = (u_1 \dots u_m)^{n-1}\left(z + (-1)^n\left ( \sum\limits_{j=1}^m u_j \right )^{n-1}\cdot \sum\limits_{j=1}^m\frac{x_j}{u_j^{n-1}}\right)
\end{equation}
of degree $m(n-1)$ in the variables $u_1\dots u_m$.

\begin{prop}
The discriminant $\disc(P(u))$ of the polynomial (\ref{homogeneous_mult}) is a homogeneous polynomial of degree $m(m(n-1)-1)^{m-1}$.
\end{prop}

\begin{proof}
Indeed, $\disc(P(u))$ is the restriction of the discriminant $\disc$ of the space of homogeneous polynomials of degree $m(n-1)$ in $m$ variables to an $(m+1)$-dimensional vector subspace parametrized by the variables $z, x_1\dots x_m$ in the coefficient space $\Cc[u_{i, \alpha}]$.
\end{proof}

Note that, as follows from Hilbert's Nullstellensatz, the aforementioned result from \cite[Theorem 3.2]{Gaiur} has the following equivalent reformulation in terms of the discriminants of homogeneous polynomials:

\begin{theorem}\label{multi_disc}
The discriminant $\disc P_x(u)$ of the homogeneous polynomial (\ref{homogeneous_mult}), which has degree $m(n-1)$ in the variables $u_1\dots u_m$, is divisible by certain powers of the monomial $zx_1 \dots x_m$ and the polynomial $p_n(z; x_1\dots x_m)$.
\end{theorem}

Let the monomial $zx_1 \dots x_m$ appear with multiplicity $\alpha$, and the polynomial $p_n(z; x_1\dots x_m)$ with multiplicity $\beta$. Then, the degree counting gives the equation
\begin{equation}\label{discriminant_degree}
\alpha(m+1) + \beta n^{m-1} = m(mn - (m+1))^{m-1}.
\end{equation}

\begin{example}
From the conditions of (\ref{discriminant_degree}), it follows that the polynomial $p_n$ does not always appear in the discriminant to the first degree. For example, this is not the case when $m = 3$ and any odd $n$.
\end{example}

\section{$n$-Valued Structures on Representations of Groups}\label{characters}

In \cite[Theorem 5]{Buchstaber}, an $n$-valued group structure is introduced on the set $\{\rho_1\dots \rho_k\}$ of finite dimensional irreducible unitary representations of a finite group $G$. We will recall this construction here.

Let ${a_{jk}^{\ell}}$ be the multiplicities in the decomposition of the tensor product $\rho_j \otimes \rho_k$ into a sum of irreducible representations a finite group $G$:
$$\rho_j \otimes \rho_k = \sum\limits_{\ell = 0}^k a_{jk}^{\ell} \rho_{\ell}.$$

\begin{prop}
Let $d_j = \dim \rho_j$, $n = \text{lcm}(d_j d_k \mid j, k = 0\dots k)$, and $m_{jk}^{\ell} = n a_{jk}^{\ell} d_{\ell} / (d_j d_k)$. Then, on the set of normalized irreducible representations $\{\rho_j / \dim \rho_j\}$, there is an $n$-valued group structure $\G^{\rep}_{\otimes}$, such that the element $\rho_\ell / \dim \rho_\ell$ appears in the $n$-multiset $\rho_j / \dim \rho_j \ast \rho_k / \dim \rho_k$ with multiplicity $m_{jk}^{\ell}$.
\end{prop}

\begin{prop}
The $n$-valued group $\G^{\rep}_{\otimes}$ is commutative.
\end{prop}

\begin{proof}
For any square matrices $A$ and $B$ of fixed dimensions, the matrix $A \otimes B$ is similar to the matrix $B \otimes A$ via some permutation matrix $U$, which does not depend on $A$ and $B$ \cite[Corollary 4.3.10]{Horn}. Hence, unitary representations $\phi \otimes \psi$ and $\psi \otimes \phi$ are equivalent for any $\phi$ and $\psi$ from $\G^{\rep}_{\otimes}$.
\end{proof}

Now, let us replace the tensor product with the Kronecker sum. Let $\{\psi_j\}$ be the set of all (including reducible) finite-dimensional representations of the group $G$ with $\dim \psi_j = d_j$. For the corresponding characters ${\chi_j}$ of the representations, we have
\begin{equation}
\begin{aligned}
\chi(\psi_j \boxplus \psi_k) &= \tr(\psi_j \otimes \id_{d_k} + \id_{d_j} \otimes \psi_k) \
&= d_k \chi_j + d_j \chi_k.
\end{aligned}
\end{equation}
Equivalently,
\begin{equation}\label{character_equality}
\frac{\chi_j}{d_j} + \frac{\chi_k}{d_k} = \frac{\chi(\psi_j \boxplus \psi_k)}{d_k d_j}.
\end{equation}
Thus, for arbitrary representations $\psi_j$ and $\psi_k$, by the correspondence between representations and characters and using the equality (\ref{character_equality}), we obtain the formal equality
\begin{equation}\label{rep_equality}
\frac{\psi_j}{d_j} + \frac{\psi_k}{d_k} = \frac{\psi_j \boxplus \psi_k}{d_k d_j}.
\end{equation}
Let us define
$$\G^{\rep}_{\boxplus} = \{\psi / \dim \psi \mid \psi \text{ is a character of a representation of the group } G\}.$$

\begin{theorem}\label{Kronecker_sum_group}
Let $G$ be a finite group. The set $\G^{\rep}_{\boxplus}$, with the operation (\ref{character_equality}), neutral element 0, and formal inverse $\inv(\psi / \dim \psi) = -\psi / \dim \psi$, is an abelian group. Moreover, $\G^{\rep}_{\boxplus} \cong \Z^N$, where $N$ is the number of conjugacy classes of the group $G$.
\end{theorem}

\begin{proof}
From equality (\ref{rep_equality}), it follows that the set $\G^{\rep}_{\boxplus}$ is closed under the operation in question. From the associativity of the Kronecker sum, the operation is associative. Thus, the set $\G^{\rep}_{\boxplus}$ is endowed with the structure of a monoid. In the definition of the Kronecker sum $A \boxplus B$, the matrices $A \otimes I$ and $I \otimes B$ are involved. As mentioned in this section, the matrix $A \otimes I$ is similar to the matrix $I \otimes A$ via the permutation matrix $U$. Similarly, $I \otimes B$ is similar to $B \otimes I$ via the same matrix $U$. Therefore, the matrices $A \boxplus B$ and $B \boxplus A$ are similar via the matrix $U$.

Now, having a commutative monoid $\G^{\rep}_{\boxplus}$, we apply the Grothendieck construction and obtain the desired structure of an abelian group.

As is well known, every representation of a finite group $G$ over the field $\Cc$ decomposes into a sum of irreducible representations. Thus, the set $$\{\rho / \dim \rho \mid \rho \text{ is an irreducible representation}\}$$ forms the set of generators of the group $\G^{\rep}_{\boxplus}$. The absence of relations follows from the orthogonality of the characters of the irreducible representations.
\end{proof}

Note that n-valued groups on the set of finite-dimensional irreducible unitary representations of certain compact Lie groups were constucted in \cite{Buchstaber}.

\section{Groups $\G^m_n$}\label{G^m_n}

We introduce the following:

\begin{defi}\label{cartesian}
The direct product of the $n$-valued and $m$-valued groups $G$ and $H$ with operations denoted by $\ast$ is called the $mn$-valued group $G \times H$ with the operation $$\star: (G \times H)^2 \to \Sym^{mn}(G \times H),$$ such that
$$ (g_1, h_1) \star (g_2, h_2) = [\left( (g_1 \ast g_2)_{r_1}, (h_1 \ast h_2)_{r_2} \right) \mid r_1 = 1, \dots, m, r_2 = 1, \dots, n], $$
the identity is $(e_G, e_H)$, and the inverse is
$$ \inv(g, h) = (\inv(g), \inv(h)). $$
\end{defi}

In other words, the multiset $(g_1, h_1) \star (g_2, h_2)$ is the Cartesian product of the multisets $g_1 \ast g_2$ and $h_1 \ast h_2$ and the identity and inverse are component-wise.

\begin{prop}
Definition \ref{cartesian} is correct.
\end{prop}

\begin{proof}
We verify associativity; the other axioms are verified similarly. Let’s compare the first components of the multivectors $A = (\alpha_1 \star \alpha_2) \star \alpha_3$ and $B = \alpha_1 \star (\alpha_2 \star \alpha_3)$ for $\alpha_i = (g_i, h_i) \in G \times H$. They represent elements of the multiset $a = (g_1 \ast g_2) \ast g_3 = g_1 \ast (g_2 \ast g_3)$. Similarly, the second components are elements of the multiset $b = (h_1 \ast h_2) \ast h_3 = h_1 \ast (h_2 \ast h_3)$. Each of the multisets $A$ and $B$ is then the Cartesian product of the multisets $a$ and $b$.
\end{proof}

\begin{example}\label{G_n^m}
Consider the Cartesian power $\G_n^m$ of the group $\G_n$. For positive integer numbers $m$ and $n$, we obtain the following $n^m$-valued group multiplication on the set $\Cc^m$ of complex vectors of dimension $m$: in the multiset $(x_1, \dots, x_m) \ast (y_1, \dots, y_m)$, there are $n^m$ vectors of the form
$$ \left( (\sqrt[n]{x_1} + \epsilon^{r_1} \sqrt[n]{y_1})^n, \dots, (\sqrt[n]{x_m} + \epsilon^{r_m} \sqrt[n]{y_m})^n \right),$$ where $r_1, \dots, r_m = 1, \dots, n.$
The pair $(\Cc^m, \ast)$, together with the identity $(0, \dots, 0)$ and the inverse
$$ \inv(x_1, \dots, x_m) = (-1)^n(x_1, \dots, x_m), $$
forms the $n^m$-valued group $\G_n^m$.
\end{example}

\begin{prop}
\begin{enumerate}[(i)]
\item The Cartesian product of two commutative groups is a commutative group.

\item The Cartesian product of two coset groups (see Definition \ref{coset_defi}) is a coset group.
\end{enumerate}
\end{prop}

\begin{proof}
\begin{enumerate}[(i)]
\item Obvious.

\item Consider two coset groups on the sets of orbits $G/H_1$ and $R/H_2$ for 1-valued groups $G$, $R$, and some finite subgroups $H_1$ and $H_2$ in the automorphism groups $\Aut(G)$ and $\Aut(R)$, respectively. By Definition \ref{cartesian} of Cartesian products of groups, we have
\begin{equation}\label{coset_product}
(g_1 H_1, r_1 H_2) \star (g_2 H_1, r_2 H_2) = \left[ \pi_1\left( g_1 \phi(g_2) \right), \pi_2\left( r_1 \psi(r_2) \right) \mid \phi \in H_1, \psi \in H_2 \right].
\end{equation}
On the other hand, the group $H_1 \times H_2 = \{\phi \times \psi \mid \phi \in H_1, \psi \in H_2\}$ naturally acts on the direct product $G \times R$ of the two groups. The resulting orbit space $G/H_1 \times R/H_2$ is equipped with the structure of a coset group with multiplication (\ref{coset_product}).
\end{enumerate}
\end{proof}

From this proposition, we conclude:

\begin{prop}
The group $\G_n^m$ is a commutative coset group.
\end{prop}

\begin{prop}\label{prop_action}
Let $G$ be a commutative $n$-valued group acting on a set $X$. Then, for each positive integer number $m$, the following left action of the Cartesian power $G^m$ of the group $G$ on the set $X$ is defined:
$$ (g_1, \dots, g_m) \circ x = (g_1\ast \dots\ast g_m) \circ x, $$
where $(g_1\ast \dots \ast g_m) \circ x$ denotes the union of the multisets $g \circ x$ for all $g \in g_1 \ast \dots \ast g_m$. (By the associativity of the operation $\ast$, the multiset $g_1 \ast \dots \ast g_m$ is independent of the order of parentheses.)
\end{prop}

\begin{proof}
Consider the case $m = 2$ (the general case is similar). By the definition of multiplication in the direct product of groups, we have
\begin{equation}\label{eq5}
((g_1, h_1) \star (g_2, h_2)) \circ x = ((g_1 \ast g_2) \times (h_1 \ast h_2)) \circ x.
\end{equation}
By the definition of the action $\circ$,
\begin{equation}
(g_1, h_1) \star ((g_2, h_2) \circ x) = (g_1 \ast h_1 \ast g_2 \ast h_2) \circ x.
\end{equation}
Multiset (\ref{eq5}) is
$$ ((g_1 \ast g_2) \ast (h_1 \ast h_2)) \circ x = (g_1 \ast h_1 \ast g_2 \ast h_2) \circ x, $$
due to the commutativity and associativity of the operation $\ast$.
\end{proof}

From Proposition \ref{prop_action} for $X = G$ and the action of the group $G$ on itself by left shifts, we obtain the following corollary:

\begin{prop}\label{G^n_action_on_G}
Let $G$ be an $n$-valued group. There is a right action of the group $G^n$ on the set $G$, compatible with the operation of the group $G$, that is,
$$ (h_1, \dots, h_n) \circ (g_1 \ast g_2) = ((h_1, \dots, h_n) \circ g_1) \ast g_2. $$
\end{prop}

From Proposition \ref{G^n_action_on_G}, we deduce the following {\it dynamic interpretation} for the polynomials $p_n(z; x_1, \dots, x_m)$:

\begin{theorem}\label{action}
Homogeneous symmetric polynomials $p_n(z; x_1, \dots, x_m)$, defined by equation (\ref{gen_p_n}), describe the $n^m$-valued dynamics of the action of the $n^m$-valued group $\G_n^m$ (see Example \ref{G_n^m}) on the set $\G_n$ (see Definition \ref{G_n_defi}) compatible with the operation of the group $\G_n$ in the sense of Proposition \ref{G^n_action_on_G}.
\end{theorem}

\newpage

\bibliographystyle{mystyle}
\bibliography{data}

\vspace{2em}
\noindent
\textit{Victor Buchstaber}\\
Steklov Mathematical Institute of Russian Academy of Sciences\\
Email: \texttt{buchstab@mi-ras.ru}

\vspace{2em}
\noindent
\textit{Mikhail Kornev}\\
Steklov Mathematical Institute of Russian Academy of Sciences\\
Email: \texttt{mkorneff@mi-ras.ru}

\end{document}